\documentclass[reqno]{amsart}
\usepackage{amsmath,amssymb,verbatim,hyperref,mathrsfs,stmaryrd}%
\providecommand{\arxiv}[2][]{\href{http://www.arXiv.org/abs/#2}{arXiv:#2}}

\renewcommand{\geq}{\geqslant}
\renewcommand{\leq}{\leqslant}

\chardef \atcode = \the \catcode `\@
\catcode `\@ = 11
\catcode246=13  \def^^f6{\"o}  
\catcode `\@ = \the \atcode
\newenvironment{rajout}{}{}

\newcommand{\br}{\begin{rajout}}
\newcommand{\er}{\end{rajout}}
\newenvironment{mycomment}{\footnotesize \bf Comments: }{}
\newcommand{\bc}{\begin{mycomment}}
\newcommand{\ec}{\end{mycomment}}

\numberwithin{equation}{section}

\newtheorem{theorem}{Theorem}[section]
\newtheorem{lemma}[theorem]{Lemma}
\newtheorem{proposition}[theorem]{Proposition}
\newtheorem{corollary}[theorem]{Corollary}
\theoremstyle{definition}

\theoremstyle{remark}
\newtheorem{remark}[theorem]{Remark}

\newcommand{\closure}[1]{\operatorname{cl}#1}
\newcommand{\tho}{d_T}%
\newcommand{\hil}{d_H}%
\newcommand{\C}{\mathscr{C}}
\newcommand{\A}{\mathcal{A}}

\newcommand{\R}{\mathbb{R}}
\newcommand{\N}{\mathbb{N}}
\newcommand{\Z}{\mathbb{Z}}
\newcommand{\leqc}{\leq_{C}}
\newcommand{\NEW}[1]{{\em #1}\index{#1}}
\newcommand{\NEWINDEX}[2]{{\em #1}\index{#1#2}}

\newcommand{\interior}{\mrm{int}\,}
\newcommand{\Cint}{\interior C}

\newcommand{\mrm}[1]{\text{\rm #1}}
\newcommand{\set}[2]{\{#1\mid\,#2\}}

\newcommand{\con}{\operatorname{conv}}

\newcommand{\PF}{\mrm{(F)}}

\newcommand{\id}{\mrm{Id}}
\newcommand{\seq}[2]{\langle #1\mid #2\rangle}
\newcommand{\stf}[2]{M(#2\, /\,#1)}
\newcommand{\sbf}[2]{m(#2\, /\,#1)}
\def\M(#1/#2){M(#1\,/\,#2)}
\def\m(#1/#2){m(#1\,/\,#2)}

\makeatletter
\long\def\tlist@if@empty@nTF #1{%
\expandafter\ifx\expandafter\\\detokenize{#1}\\%
\expandafter\@firstoftwo
\else
\expandafter\@secondoftwo
\fi
}

\newcommand{\norm}[2][]{%
\|#2\|\tlist@if@empty@nTF{#1}{}{_{#1}}%
}

\newcommand{\bonsall}[2][]{%
\tilde{r}(#2
\tlist@if@empty@nTF{#1}{}{\,,#1}%
)
}
\newcommand{\cw}[2][]{
\operatorname{cw}%
(#2
\tlist@if@empty@nTF{#1}{}{\,,#1}%
)
}
\newcommand{\cspr}[2][]{
r(#2
\tlist@if@empty@nTF{#1}{}{\,,#1}%
)
}
\newcommand{\eigenvalspr}[2][]{
\hat{r}(#2
\tlist@if@empty@nTF{#1}{}{\,,#1}%
)
}
\newcommand{\cczech}[2][]{
\check{r}(#2
\tlist@if@empty@nTF{#1}{}{\,,#1}%
)
}

\newcommand{\mesname}{\nu}
\newcommand{\mes}[2][]{
\mesname(#2
\tlist@if@empty@nTF{#1}{}{\,,#1}%
)
}

\newcommand{\essp}[2][]{
\rho(#2
\tlist@if@empty@nTF{#1}{}{\,,#1}%
)
}

\newcommand{\esspalt}[2][]{
\tilde{\rho}(#2
\tlist@if@empty@nTF{#1}{}{\,,#1}%
)
}
\newcommand{\darbo}{\alpha}
\newcommand{\contract}[2][]{
\darbo(#2
\tlist@if@empty@nTF{#1}{}{\,,#1}%
)
}
\makeatother

\newcommand{\trn}{\interleave}
\makeatletter
\def\laba#1{{\rm (\@alph {#1})}}
\makeatother
\makeindex
\begin{document}
\title[A Collatz-Wielandt characterization of the spectral radius]{A Collatz-Wielandt characterization of the spectral radius of order-preserving homogeneous maps on cones}
\author{Marianne Akian}
\address{Marianne Akian, INRIA \& Centre de
Math\'ematiques Appliqu\'ees (CMAP), UMR 7641 CNRS,
\'Ecole Polytechnique, 91128
Palaiseau, France.}
\email{marianne.akian@inria.fr}
\author{St\'ephane Gaubert}
\address{St\'ephane Gaubert, INRIA \& Centre de
Math\'ematiques Appliqu\'ees (CMAP), UMR 7641 CNRS,
\'Ecole Polytechnique, 91128
Palaiseau, France.}
\email{stephane.gaubert@inria.fr}
\author{Roger Nussbaum}
\address{Roger Nussbaum, Mathematics Department, Hill Center,
Rutgers University, 110 Frelinghuysen Road,
Piscataway, New Jersey, U.S.A. 08854-8019}
\email{nussbaum@math.rutgers.edu}
\date{\today}
\thanks{The first two authors were partially supported by the PGMO Programme of FMJH and EDF, and by the programme ``Ing{\'e}nierie Num{\'e}rique \& S{\'e}curit{\'e}'' of the French National Agency of Research, project ``MALTHY'', number ANR-13-INSE-0003. 
The third author was partially supported by NSFDMS 0701171 and by NSFDMS 1201328.}

\begin{abstract}
Several notions of spectral radius arise in the study of nonlinear
order-preserving positively homogeneous self-maps of cones in Banach
spaces.  We give conditions that guarantee that all these notions lead to
the same value. In particular, we give a Collatz-Wielandt type
formula, which characterizes the growth rate of the orbits in terms of
eigenvectors in the closed cone or super-eigenvectors in the interior
of the cone. This characterization holds when the cone is normal and
when a quasi-compactness condition, involving an essential spectral
radius defined in terms of $k$-set-contractions, is satisfied.
Some fixed point theorems for non-linear maps on cones are derived as intermediate results. We finally apply these results to show that non-linear spectral radii commute with respect to suprema and infima of families of order preserving maps satisfying selection properties.
\end{abstract}
\maketitle

\section{Introduction}
Non-linear self-maps of a Banach space preserving the order induced by a 
(closed, convex, and pointed) cone 
arise in a number of fields,
including population dynamics~\cite{perthame},
entropy maximization and scaling problems~\cite{menonSchneider69,borweinlewisnussbaum}, renormalization operators and fractal diffusions~\cite{sabot,metz,LinsNus08}, mathematical economy~\cite{morishima}, 
mathematical biology~\cite{angeli08}, optimal filtering and optimal control~\cite{bougerol}, and zero-sum games~\cite{kolokoltsov92,rosenbergsorin,neymansurv}, the latter being related with tropical geometry~\cite{AGGut10}. 
They turn out to share many of the 
features of nonnegative matrices, as shown by
a number of works establishing non-linear
analogues of classical results of Perron-Frobenius theory,
including~\cite{krut48,birkhoff57,birkhoff62,hopf,krasnoselskii,birkhoff67,potter,bushell73,bushell86,nussbaum88,nussbaummemoir89,krause,nussbaumlunelmemoir,arxiv1,AGLN,hujiang10a,hujiang10b,GV10}. See~\cite{lemmensnussbaum} for a recent overview.

A classical problem, for a map $f$ leaving invariant a subset $D$ of a Banach space $X$, is  to characterize the (maximal) \NEW{growth rate} of the orbits of $f$:
\[
\cspr[D]{f}:=\sup_{x\in D}\limsup_{k\to\infty} \|f^k(x)\|^{1/k}  
\enspace, 
\]
where $f^k:=f\circ\dots\circ f$ denotes the $k$th iterate of $f$.

When $f$ is \NEW{positively homogeneous} (meaning that $f$ commutes with the 
product with a positive constant), and when $f$ preserves
the order induced by a cone $C$, we may look for non-linear
\NEWINDEX{sub}{eigenvector!sub} and \NEWINDEX{super-eigenvectors}{eigenvector!super}, $u,v\in C\setminus\{0\}$, satisfying respectively
\begin{align}\label{sp-e-dualeigen}
f(u)\leq \lambda u \enspace,\qquad\qquad
f(v)\geq \mu v\enspace, 
\end{align}
where $\lambda,\mu>0$; see Sections~\ref{sp-sec-intro-hilbert} and~\ref{sp-sec-spec} for more background.
In particular, if the cone is \NEWINDEX{normal}{normal!cone}, it is easily verified that if $x$ is such that 
\[a v\leq x\leq b u
\]
for some positive constants $a$ and $b$,
then 
\begin{align}\label{e-2bound}
\mu \leq \limsup_{k\to\infty} \|f^k(x)\|^{1/k}   \leq \lambda  
\enspace .
\end{align}
Furthermore, a simple argument shows that the first inequality 
in~\eqref{e-2bound}
holds even if the cone is not normal. 
See Lemma~2.2 in~\cite{LN13}. 

A fortiori, the same conclusion persists if we require $u$ or $v$
to be eigenvectors, i.e., to satisfy the equality in~\eqref{sp-e-dualeigen}.
It is natural to ask whether the bounds obtained in this way are tight.
The main result of this paper shows that the answer is positive,
under rather general circumstances.
\begin{theorem}[Generalized Collatz-Wielandt Theorem]\label{sp-th-gencw}
Let $C$ be a normal cone with non-empty interior in a Banach space,
and let $f$ be a positively homogeneous self-map of $C$ that
preserves the partial order induced by $C$. Suppose that $f$ is uniformly continuous on every bounded subset of $C$.
Then, we have
\begin{align}
\cspr[C]{f}
&= \inf\set{\lambda >0}{\exists u\in\interior C,\; f(u)\leq \lambda u} \label{sp-e-init2} \enspace .
\end{align}
If, in addition, the cone essential spectral radius of $f$ is strictly less
than $\cspr[C]{f}$, we also have
\begin{align}
\cspr[C]{f}
 &=
 \max\set{\mu\geq 0 }{\exists v\in C\setminus\{0\},\; f(v)= \mu v}\label{sp-e-init} \enspace .
\end{align}
\end{theorem}
The notion of \NEW{cone essential spectral radius} 
appearing in Theorem~\ref{sp-th-gencw} is a non-linear extension of the notion
of essential spectral radius of linear maps
introduced in~\cite{nussbaum70}.
It is defined using generalized measures of non-compactness and 
$k$-set-contractions
(Section~\ref{sp-sec-spec}).
Note that when some iterate of the map $f$ is compact, the cone essential spectral radius of $f$ is zero.
Then, the assumption of the theorem is satisfied as soon as $\cspr[C]{f}>0$.

Theorem~\ref{sp-th-gencw} follows from Theorem~\ref{sp-supeigenth0} below. 
Note the lack of symmetry between~\eqref{sp-e-init} and~\eqref{sp-e-init2}:
we use the notation max in~\eqref{sp-e-init} to indicate that the set has
a maximum element, whereas the infimum in~\eqref{sp-e-init2} is not attained
in general. Thus, this theorem states in particular
that there exists $v\in C\setminus\{0\}$ such that $f(v)=\cspr[C]{f}v$.
In other words, the maximal growth rate of the orbits
coincides with the maximal non-linear eigenvalue of $f$ associated
to an eigenvector in the (closed) cone $C$.

The classical theorem of Collatz and Wielandt concerns the case
in which $X=\R^n$, $C=\R_+^n$ and $f$ is linear. Nussbaum extended
this result in~\cite[Th.~3.1]{nussbaum86} to non-linear continuous self-maps of $\R_+^n$ preserving the partial order induced by the cone $\R_+^n$. 
Theorem~\ref{sp-th-gencw} should also be compared with a result of
Mallet-Paret and Nussbaum, who showed
that under the assumption
that the cone essential spectral radius of $f$ is strictly less
than $\cspr[C]{f}$, 
and $C$ is a closed, proper cone (not necessarily normal),
there exists an integer $m$ and a vector $v$ such
that $f^m(v)=(\cspr[C]{f})^m v$, see~\cite[Th.~3.1]{Nuss-Mallet}.
An interest of a Collatz-Wielandt type characterization
lies in its {\em strong duality} nature: 
the eigenvector $v\in C$ and the super-eigenvector $u\in \operatorname{int}C$
allow one to bound the growth rate from above and from below.

In the course of proving Theorem~\ref{sp-th-gencw}, we establish several
non-linear fixed-point results of independent interest. 
In particular, Theorem~\ref{sp-th-add1-new}
which completes the previously mentioned result of~\cite{Nuss-Mallet},
shows that under the same compactness
assumption (on the cone essential spectral radius),  and when the cone ordering
of the Banach space $X$ induces a lattice, but the cone is not necessarily 
normal, there exists a vector $v\in C\setminus\{0\}$
such that $f(v)\geq \cspr[C]{f}v$. 
Corollary~\ref{sp-proper-fix} gives further assumptions under which $v$ can be chosen so that the equality holds. 

In Section~\ref{sp-sec-X}, we consider
the situation in which the map $f$ is defined over the whole Banach
space and $C$ is reproducing, meaning that $X=C-C$.
Then, we show that the growth rate of the orbit of a vector $x$
is maximal when $x\in C$ or $x\in -C$, i.e.,
\begin{align}\label{sp-e-morphism}
\cspr[X]{f} = \max(\cspr[C]{f},\cspr[-C]{f}) \enspace .
\end{align}
We also show (Theorem~\ref{sp-prop-add1}) that, again under a compactness assumption, $0$ is the unique fixed point of $f$ in $X$ if and only if $\cspr[X]{f}<1$.
In other words, the uniqueness of the fixed point implies every orbit
converges to this fixed point with a geometric rate.
This result can be used in combination with the ones
of the companion article~\cite{AGN}, which gives
sufficient conditions to check the uniqueness of the fixed point of 
a semidifferentiable nonexpansive mapping.

As a consequence, we show
(Proposition~\ref{sp-lem-add3} and~\ref{sp-lem-inf}) that the spectral radius commutes
with suprema or infima of families of maps satisfying selection
properties. This is motivated by zero-sum games, in which Shapley
operators can be defined as infima or suprema of families
of more elementary operators. 
The present results are applicable to several problems
arising in this context.
A first application, to characterize
the convergence rate of value iteration,
has appeared in~\cite{mariannemtns}. A second application,
to derive complexity bounds for policy
iteration for zero-sum two player games,
has recently appeared in~\cite{marianne102013}. 
We refer to these references,
and also to the last section of~\cite{AGN}, for more
information on the application to games. 

We finally note that in work~\cite{nussbaummp10} which
postdates this paper, Mallet-Paret
and Nussbaum pointed out inadequacies with the
notion of cone essential spectral radius which is used
here and in~\cite{Nuss-Mallet}. 
Generalization of the present results to the setting
of~\cite{nussbaummp10} will be considered in a future
work.

The paper is organized as follows. In Sections~\ref{sp-sec-intro-hilbert}
and~\ref{sp-sec-spec}, we recall notions and results about cones,
non-linear spectral radii, and $k$-set-contractions.
In Section~\ref{sec-fredholm}, we establish auxiliary
results concerning a non-linear Fredholm-type property.
Our main results are established in Section~\ref{sp-sec-existence}--\ref{sp-sec-CW}. We apply some of them in Section~\ref{sp-spec-inf} to
establish the ``morphism'' properties of the spectral radius with
respect to the operations of supremum and infimum.

\section{Preliminary results about Hilbert's and Thompson's metric}
\label{sp-sec-intro-hilbert}
We next recall classical notions
about cones.
See \cite[Chapter 1]{nussbaum88} and \cite[Section 1]{nussbaum94}
for more background.

A subset $C$ of a real vector space $X$ 
is called a \NEW{cone} (with vertex 0) if 
$tC:=\set{tx}{x\in C} \subset C$ for all $t\geq 0$. 
If $f$ is a map from a cone $C$ of a vector space $X$ to a cone
$C'$ of a vector space $Y$,
we shall say that $f$ is (positively) \NEWINDEX{homogeneous}{!(positively)} (of degree $1$)
if $f(t y)=t f(y)$, for all $t>0$ and $y\in C$.
We say that the cone $C$ is \NEW{pointed}
if $C \cap (-C) = \{0\}$.
A convex pointed cone $C$ of $X$
induces on $X$ a partial ordering $\leqc$, which
is defined by $x\leqc y$ iff $y - x \in C$. 
If the choice of $C$ is obvious, we shall write $\leq$ instead of $\leqc$.
When $X$ is a topological vector space, we say that
$C$ is \NEW{proper} if it is closed convex and pointed.
Note that in \cite{nussbaum88}, 
a {\em cone} is by definition what we call here a proper cone.
We next recall the definition of Hilbert's\index{metric!Hilbert} and Thompson's
metrics \index{metric!Thompson} associated to a proper cone $C$ of a topological vector space $X$.

Let $x\in C\setminus\{0\}$ and $y\in X$. 
We define
$\M(y/x)$ by
\begin{align}\label{sp-def-M}
\M(y/x) := \inf \set{b \in \R}{ y \leq bx }\enspace ,
\end{align}
where the infimum of the empty set is by definition equal to
$+\infty$. %
Similarly, %
we define $\m(y/x)$ by
\begin{align}\label{sp-def-m}
\m(y/x) := \sup \set{a \in \R}{ax \leq y }\enspace,
\end{align}
where the supremum of the empty set is by definition equal to
$-\infty$. %
We have $\m(y/x)=-\M(-y/x)$ and if in addition 
$y\in C\setminus\{0\}$, $\m(y/x)=1/ \M(x/y)$ (with the convention
$1/(+\infty)=0$).
Since $C$ is pointed and closed, we have $\M(y/x)\in \R\cup\{+\infty\}$, and
$y \leq \M(y/x) x$ as soon as $\M(y/x)<+\infty$. %
Symmetrically $m(y/x)\in \R\cup\{-\infty\}$ and $\m(y/x) x \leq y$, as soon as 
$m(y/x)>-\infty$. %

We shall say that two elements $x$ and $y$ in $C$ are \NEW{comparable}
and write $x \sim_C y$  or $x \sim y$ 
if there exist positive constants $a > 0$ and $b > 0$ such that
$ax \leq y \leq bx$. 
If $x,y \in C\setminus \{0\}$ are comparable,
we define
\begin{align*}
\hil(x,y) = &  \log \M(y/x) -\log \m(y/x) \enspace ,%
\\
\tho(x,y) = &  \log \M(y/x) \vee (-\log \m(y/x))\enspace ,
\end{align*}
where we use the notation $a\vee b=\max(a,b)$.
We adopt the convention that $\hil(0,0)=\tho(0,0)=0$.
If $u \in C$, the set of elements comparable with $u$,
\[
C_u := \{x \in C \mid x \sim u\}\enspace ,
\]
is called a \NEW{part of $C$}.
If $C$ has nonempty interior $\Cint$, and $u \in \Cint$, 
then $C_u = \Cint$. 
In general, $C_u \cup \{0\}$ is a pointed convex cone
but $C_u \cup \{0\}$ is not closed. 
The map $\tho$ %
is a metric on $C_u$,
called \NEW{Thompson's metric}.
The map $\hil$ %
is called the \NEW{Hilbert projective metric}
on $C_u$. The term ``projective metric''
is justified by the following properties:
for all $x,y,z \in C_u$, $\hil(x,z) \leq \hil(x,y) + \hil(y,z)$, 
$\hil(x,y) = \hil(y,x)\geq 0$ and $\hil(x,y) = 0$
iff $y = \lambda x$ for some $\lambda > 0$.

{From} now on, we will assume that $X=(X,\norm{\cdot})$ is a Banach space.
We denote by $X^*$ the space of continuous linear forms over $X$,
and by $C^*:=\set{\psi\in X^*}{\psi(x)\geq 0\; \forall x\in C}$ 
the \NEW{dual cone} of $C$.
If $f$ is a map between two ordered sets $(D,\leq)$ and $(D',\leq)$, 
we shall say that $f$ is \NEW{order-preserving} if $f(x)\leq f(y)$ for all   
$x,y\in D$ such that $x\leq y$. Then, any element of $C^*$ is a
homogeneous and order-preserving  map from $(X,\leq_C)$ to $[0,+\infty)$.
If the cone $C$ is proper, the Hahn-Banach theorem implies 
that for all $u\in C\setminus\{0\}$, 
there exists $\psi\in C^*$ such that $\psi(u)>0$.
For such a $\psi$, we have $\psi(x)>0$ for all $x\in C_u$.
More generally, if $q:C_u\to (0,+\infty)$ is homogeneous and
order-preserving, we shall write \index{$\Sigma_u$}
\[\Sigma_u=\set{x\in C_u}{ q(x)=q(u)}
\enspace .
\]
Then, $\hil$ and $\tho$ are equivalent metrics on $\Sigma_u$.
Indeed, as shown in~\cite[Remark~1.3, p. 15]{nussbaum88}:
\begin{align}
\frac 1 2 \hil(x,y)\leq \tho(x,y) \leq \hil(x,y),\; \forall x,y\in \Sigma_u
\enspace.
\label{sp-e-equivm}
\end{align}
More precisely:
\begin{align}
1\leq \M(y/x) \leq e^{\hil(x,y)}\quad \forall x,y\in \Sigma_u\enspace.
\label{sp-mleqd}
\end{align}
To see this, let us apply $q$ to the inequality $y\leq \M(y/x)\; x$.
Using that $q$ is order-preserving and homogeneous, and
that $q(x)=q(y)$, we get $\M(y/x)\geq 1$.
By symmetry, $\M(x/y)\geq 1$, hence 
$\log \M(y/x)= \hil(x,y)-\log\M(x/y)\leq \hil(x,y)$.

We say that a cone $C$ is 
\NEW{normal} if $C$ is proper and 
there exists a constant $M$ such that $\norm{x} \leq M \norm{y}$ 
whenever $0 \leq x \leq y$. Every proper cone $C$ in a
finite dimensional Banach space $(X,\norm{\cdot})$ is necessarily normal.
We shall need the following result of Thompson.
\begin{proposition}[{\cite[Lemma~3]{thompson}}]\label{sp-prop-thompson}
Let $C$ be a normal cone in a Banach space $(X,\norm{\cdot})$.
For all $u\in C\setminus\{0\}$, $(C_u,\tho)$ 
is a complete metric space.
\end{proposition}
The next proposition follows from a general result
of Zabre{\u\i}ko, Krasnosel{$'$}ski{\u\i} and Pokorny{\u\i}~%
\cite{zabreiko} (see~\cite[Theorem~1.2 and Remarks~1.1 and 1.3]{nussbaum88}
and a previous result of Birkhoff~\cite{birkhoff62}).
When $q\in C^*$, it follows from 
Proposition~\ref{sp-prop-thompson}, together with Eqn~\eqref{sp-e-equivm} and the property that
$\Sigma_u$ is closed in the topology of the Thompson's metric $\tho$.
\begin{proposition}\label{sp-prop-birkhoff}
Let $C$ be a normal cone in a Banach space $(X,\norm{\cdot})$.
Let $u\in C\setminus\{0\}$ and let $q:C_u\to (0,+\infty)$,
be homogeneous and order-preserving with respect to $C$. 
Define $\Sigma_u=\set{x\in C_u}{ q(x)=1}$. Then,  
$(\Sigma_u,d)$ and $(\Sigma_u,\tho)$ are complete metric spaces. 
\end{proposition} 
Given $u \in C\setminus \{0\}$, we define the linear space 
\[X_u=\set{x \in X}{\exists a > 0,\; -au \leq x \leq au}
\enspace .
\]
Let $M$ and $m$ be defined as in~\eqref{sp-def-M} and \eqref{sp-def-m}.
We equip $X_u$ with the norm:
\begin{equation}\label{sp-eq6}
\norm[u]{x} = \stf{u}{x}\vee (-\sbf{u}{x}) = 
\inf \set{a > 0}{- au \leq x \leq au} \enspace .
\end{equation}

\begin{proposition}[{\cite[Proposition~1.1]{nussbaum94}}]\label{sp-prop-normequiv}
Let $C$ be a normal cone of nonempty interior
in a Banach space $(X, \norm{\cdot})$.
If $u\in \Cint$,
then
$X_u = X$ and $\norm{\cdot}$ and $\norm[u]{\cdot}$ are equivalent norms on $X$.
\end{proposition}

We say that a cone $C$ in a Banach space $(X,\norm{\cdot})$ is
\NEW{reproducing} if $X=C-C:=\set{x-y}{x, y\in C}$.
The following observation is standard. We include
a proof for the convenience of the reader.
\begin{proposition}\label{sp-normalempty}
A cone $C$ in a Banach space $(X, \norm{\cdot})$
with nonempty interior is reproducing.
\end{proposition}
\begin{proof}
If $C$ is a cone and has nonempty interior, take $u\in \Cint$.
Then, $B_\epsilon(u)=\set{z}{\norm{z-u}\leq \epsilon}\subset C$
for some $\epsilon>0$. Consider $x\in X$. If $x=0$, $x\in C$.
If $x\neq 0$, $u\pm \frac{\epsilon}{\norm{x}} x\in C$, which
implies that $\pm\frac 1 2 x +\frac{\norm{x}}{2\epsilon}u \in C$,
and so, $x=(\frac 1 2 x +\frac{\norm{x}}{2\epsilon}u)
-(-\frac 1 2 x +\frac{\norm{x}}{2\epsilon}u) \in C-C$.
\end{proof}
We shall finally need
the following well known elementary property
(see for instance~\cite{nussbaum88,thompson,bushell73,potter}):

\begin{lemma}\label{sp-lemma-nonexpan}
Let $C$ be a proper cone and $u\in C\setminus\{0\}$.
If $f:C_u\to C$ is order-preserving and homogeneous,
then $f(C_u)\subset C_{f(u)}$ and 
$f$ is nonexpansive both in Hilbert's projective metric $\hil$ and 
in Thompson's metric $\tho$.
\end{lemma}

\section{Spectral radius notions and $k$-set-contractions}
\label{sp-sec-spec}

In this section, we recall some results of~\cite{Nuss-Mallet}
concerning spectral radii of non-linear maps
and $k$-set-contractions.

Let $C$ be a cone in a Banach space $(X, \norm{\cdot})$.
If $h$ is a homogeneous map (of degree 1)
from $C$ to a normed vector space $(Y,\norm{\cdot})$, we define:
\begin{equation}\label{sp-definormh}
\norm[C]{h}:=\sup_{x\in C\setminus\{0\}} \frac{\norm{h(x)}}{\norm{x}} 
\enspace .
\end{equation}
If $h$ is continuous at point $0$, then $\norm[C]{h}<+\infty$.
Indeed, since $h(0)=0$, by continuity of $h$,
there exists $\delta>0$ such that $\norm{h(x)}\leq 1$ for all 
$x\in C$ such that $\norm{x}\leq \delta$. Hence, by homogeneity of $h$,
$\norm[C]{h}\leq 1/\delta$.

Consider now a homogeneous
map $h$ from $C$ to $C$. Following~\cite{Nuss-Mallet}, we define:
\begin{subequations}%
\begin{align}
\bonsall[C]{h} & =
\lim_{k \rightarrow \infty} \norm[C]{h^k}^{1/k} 
=\inf_{k\geq 1} \norm[C]{h^k}^{1/k}\enspace ,
\label{sp-eq31-bonsallcspr}\\
\cspr[C]{h} &= \sup_{x \in C} \mu(x) \qquad \text{ where } \mu(x)= \limsup_{k \rightarrow \infty} \norm{h^k(x)}^{1/k} \enspace,\label{sp-eq31-cspr}\\
\eigenvalspr[C]{h} & = \sup\, \set{ \lambda \geq 0}{\exists x \in C \backslash \{0\}, \; h(x) = \lambda x } \enspace ,\\
\cczech[C]{h} & =  \sup_{k \geq  1}\, (\eigenvalspr[C]{h^k})^{1/k}  \enspace.
\label{sp-eq31-cczech}
\end{align}
\end{subequations}
When $C=X$, $C$ will be omitted in the previous notation. 
The equality of the limit and the infimum in~\eqref{sp-eq31-bonsallcspr}
follows from $\norm{h^{k+\ell}}\leq \norm{h^k} \norm{h^\ell}$.
The number $\bonsall[C]{h}$ is called
\NEW{Bonsall's cone spectral radius} of $h$,
$\cspr[C]{h}$ is called the \NEW{cone spectral radius} of $h$,
and $\eigenvalspr[C]{h}$ is called the \NEW{cone eigenvalue spectral radius}
of $h$. We have the following elementary inequalities.
\begin{proposition}[{\cite[Eqn (2.9) and Prop.~ 2.1]{Nuss-Mallet}}]
\label{sp-prop-easy}
If $h$ is a homogeneous self-map of a convex pointed cone in a Banach
space, then
\[
0\leq \eigenvalspr[C]{h} \leq \cczech[C]{h} \leq \cspr[C]{h} \leq 
\bonsall[C]{h}
\enspace .
\]
\end{proposition}
The following result 
shows that the equality $\cspr[C]{h} = \bonsall[C]{h}$ holds
in several common situations.
\begin{theorem}[{\cite[Theorem~2.2, the subsequent remark and Theorem~2.3]{Nuss-Mallet}}]
\label{sp-prop-nus1}
Let $C$ be a proper cone in a Banach space $(X,\norm{\cdot})$. Let $h:C\to C$ be a continuous and homogeneous map.
The equality
\begin{equation}
 \cspr[C]{h}=\bonsall[C]{h}\end{equation}
holds if either $C$ is normal and $h$ is order-preserving with 
respect to $C$, or $h$ is linear
or there exists $m\geq 1$ such that $h^m$ is compact.
\end{theorem}
If $C$ is a proper cone and $h:C\to C$ is continuous and homogeneous,
an example in~\cite{Nuss-Mallet} shows that it may happen that
$\cspr[C]{h}<\bonsall[C]{h}$. In a recent work~\cite{gripenberg},
Gripenberg showed that the same may happen
even if $h$ is required in addition to be order preserving.
We now introduce the notions
of \NEW{$k$-set-contraction} and 
\NEW{essential spectral radius},
which rely on measures of noncompactness.

A map $\mesname$ from the set of bounded subsets of $X$ to the set of real
nonnegative numbers is called a \NEW{homogeneous generalized measure of noncompactness} or, for brevity, a homogeneous generalized MNC, if for all bounded subsets $A,B$ of $X$ and for all
scalars $\lambda$,
\begin{subequations}
\begin{align}
\mesname(A) = 0 &\Leftrightarrow \closure{A} \text{ is compact}
 \label{sp-gmnc1}\\
\mesname(A+B) &\leq \mesname(A) + \mesname (B) \label{sp-gmnc2}\\
\mesname(\closure{\con(A)}) &= \mesname (A) \label{sp-gmnc3} \\
\mesname(\lambda A) &= |\lambda| \mesname(A) \label{sp-gmnc5}\\
A\subset B&\implies \mesname(A)\leq \mesname(B) \enspace .
\label{sp-gmnc6}
\end{align}
\end{subequations}
We use the notation $\closure A$ for the \NEW{closure}
of $A$ and $\con{A}$ for the \NEW{convex hull} of $A$.
Note that our definition is slightly more general
than that in~\cite{Nuss-Mallet}, which requires
set-additivity for $\mesname$, namely, 
that $\mesname(A\cup B)=\mesname(A)\vee\mesname(B)$ holds for all bounded
sets $A,B$. 
However, our definition of homogeneous, generalized MNC is the same as that in~\cite{MPN11}.   The question of set-additivity for a homogeneous, generalized MNC $\mesname$ is discussed at length in Theorem 2.8 in~\cite{MPN11}.  The property of set-additivity is sometimes convenient but will play no role here.  In fact we shall only need a special case of set-additivity which is proved in Proposition~2.1 in~\cite{MPN11} or can easily be derived directly from the above definition.
\begin{proposition}[Proposition~2.1 of~\cite{MPN11}]\label{sp-prop-mes-simple}
Let $\mesname$ denote a homogeneous generalized MNC
on $X$. Then, for all bounded subsets $A,B$ of $X$,
\begin{align}
\mesname(A\cup B)=\mesname(A) & \mrm{ if } B \mrm{ is relatively compact.}
\label{sp-gmnc7}
\end{align}
\end{proposition}
\if{
\begin{proof}
Observe first that for all bounded subsets $A$ of $X$, and
for all $x\in X$, $\mesname(A+\{x\})=\mesname(A)$.
Indeed, since $\{x\}$ is compact, it follows from~\eqref{sp-gmnc2}
that $\mesname(A+\{x\})\leq \mesname(A)+\mesname(\{x\})=\mesname(A)$. 
The same argument shows that $\mesname(A)=\mesname(A+\{x\}-\{x\})\leq \mesname(A+\{x\})$,
and so $\mesname(A+\{x\})=\mesname(A)$. 

Let us now fix $a_0\in A$. Then, $A+(B\cup\{a_0\})-\{a_0\}\supset A\cup B$.
By~\eqref{sp-gmnc6}, $\mesname(A\cup B)\leq \mesname(A+(B\cup\{a_0\})-\{a_0\})
=\mesname(A+(B\cup\{a_0\}))$, and by~\eqref{sp-gmnc1} and~\eqref{sp-gmnc2}, 
$\mesname(A+(B\cup\{a_0\}))\leq \mesname(A)+\mesname(B\cup\{a_0\})=\mesname(A)$.
Hence, $\mesname(A\cup B)\leq \mesname(A)$.
The opposite inequality follows from~\eqref{sp-gmnc6}.
\end{proof}}\fi

For every bounded subset $A$ of $X$,
let $\darbo(A)$ denote the infimum
of all $\delta>0$ such that there
exists an integer $k$ and
$k$ subsets $S_1,\ldots,S_k\subset A$
of diameter at most $\delta$,
such that $A = S_1\cup\cdots\cup S_k$.
The map $\darbo$, introduced by Kuratowski and further studied by Darbo
(see~\cite{Nuss-Mallet} for references), is a particular case
of a homogeneous generalized 
MNC.
It satisfies in addition, for all bounded sets $A,B\subset X$,
\begin{align}
\darbo(A\cup B) &= \darbo (A) \vee \darbo (B) 
\enspace .\label{sp-gmnc4-alpha} 
\end{align}

As another example, consider the Banach space $X=\C (W)$ 
of continuous functions from a compact metric space $(W,d)$ to $\R$.
For all bounded subsets $A$ of $X$ and $\delta>0$, define:
\begin{subequations}
\begin{align}
 \omega_\delta(A)&:=\sup\set{|x(t)-x(s)|}{x\in A,\text{ and } t,s\in W 
\text{ satisfying } d(s,t)\leq \delta}
\end{align}
and the ``modulus of equicontinuity''
\begin{align}
\omega(A)=\inf_{\delta>0}\omega_\delta (A)\enspace.
\end{align}
\end{subequations}
The map $\omega$ is a homogeneous generalized MNC
on $X$ (see \cite{Nuss-Mallet}).

If $h$ : $D \subset X \rightarrow X$ is a map, 
we define
\[
\mes[D]{h} =
\inf\set{\lambda > 0}{\mes{h(A)} \leq \lambda \mes{A},\; \mrm{ for all bounded sets } A \subset D}
\enspace .
\]
If in addition $h(D) \subset D$, we define :
\[
\essp[D,\mesname]{h} = \lim_{k \rightarrow \infty} (\mes[D]{h^k})^{1/k} = \inf_{k \geq 1} \mes[D]{h^k}^{1/k}
\enspace .
\]
If $C$ is a cone and $h$: $C \rightarrow C$
is homogeneous and Lipschitz continuous with constant $\kappa$, 
then $\contract[C]{h} \leq \kappa$.
A general map $h$: $D \subset X \rightarrow X$, such that 
$\mes[D]{h} \leq k < 1$ is called a \NEW{$k$-set-contraction}
with respect to the homogeneous generalized MNC
$\mesname$.
If $C$ is a cone and $h$ : $C \rightarrow C$ is homogeneous, $\essp[C,\mesname]{h}$ 
is called the \NEW{cone essential spectral radius} of $h$
associated to the homogeneous generalized 
MNC
$\mesname$.
In general $\mes[C]{\cdot}$ depends on the homogeneous generalized MNC
$\mesname$. 
When $\mesname$ and $\mesname'$ are equivalent in the sense
that there exist positive constants $m$ and $M$, such that for all
bounded subsets $A$ of $C$:
\[ m \mesname(A)\leq \mesname'(A)\leq M\mesname(A)\enspace,\]
then one can prove that $\essp[C,\mesname]{\cdot}=\essp[C,\mesname']{\cdot}$.
When $X=\C(W)$ as above, this property applies in particular to the modulus of equicontinuity
 $\omega$ and to the Kuratowski-Darbo MNC
$\alpha$,
which satisfy $\alpha(A)\leq \omega(A)\leq 2\alpha(A)$, for all
bounded subsets of $\C(W)$, see~\cite[Theorem~1]{nussbaum71}.

However, if the homogeneous, generalized MNC's $\mesname$ and $\mesname'$ are not equivalent, examples are given in~\cite{nussbaummp10} and~\cite{nussbaummp11} to show that it may happen that $\essp[C,\mesname]{\cdot}$ is unequal to $\essp[C,\mesname']{\cdot}$.  Because of this difficulty (among others), a different definition of the cone essential spectrum is given in~\cite{nussbaummp10}.  If we temporarily label this definition $\esspalt[C]{h}$,
one can prove, with the aid of Proposition 3.1 in~\cite{nussbaummp10} that 
$\esspalt[C]{h}\leq \essp[C,\mesname]{h}$
for any homogeneous, generalized MNC $\mesname$.

We shall defer treatment of $\esspalt[C]{h}$ to a latter
paper. Here we shall take a fixed, homogeneous, general MNC 
$\mesname$; and for notational simplicity we shall write $\essp[C]{\cdot}$
instead of $\essp[C,\mesname]{\cdot}$.

\begin{theorem}[{\cite[Th.~3.1]{Nuss-Mallet}}]\label{sp-th-nus2}
Let $C$ be a proper cone in a Banach space $(X,\norm{\cdot})$, and
$h:C\to C$ be a map that is continuous, homogeneous,
and order-preserving with respect to $C$. 
Suppose that for some $m\geq 1$, 
\[
(\mes[C]{h^m})^{1/m}< \cspr[C]{h} 
\enspace .
\]
Then, there exists a vector $x_m \in C\setminus\{0\}$ satisfying
\begin{align}
h^m(x_m)= (\cspr[C]{h})^m x_m
\enspace .
\label{sp-th-nus21} 
\end{align}
\end{theorem}
It is a special case of Theorem 3.4 in~\cite{nussbaummp10} that, under the hypotheses of Theorem 3.4 above, we have $\bonsall[C]{h}=\cspr[C]{h}$.

This result implies that, under the same assumptions
\begin{align}
\cspr[C]{h}=  (\eigenvalspr[C]{h^m})^{1/m} = \cczech[C]{h}\enspace ,
\label{sp-coro-33}
\end{align}
in particular,
\[
\essp[C]{h}< \cspr[C]{h}\implies  \cspr[C]{h}=\cczech[C]{h}\enspace ,
\]
see~\cite[Cor.~3.3]{Nuss-Mallet}.
\\
\section{A nonlinear Fredholm-type property}\label{sec-fredholm}
We now introduce a nonlinear Fredholm-type property. We also refer
the reader to~\cite{AGN} for more insight
on this property; it is used there to establish uniqueness results for the fixed
point of a nonexpansive map defined on a subset of a Banach space.

If $(X,\norm{\cdot})$ and $(Y,\norm{\cdot})$ are Banach spaces,
$D$ is a subset of $X$, and $h:D\to Y$ is
a map, we shall say that
$h$ has {\em Property~\PF} when
\begin{enumerate}
\item[(F)]any sequence 
$\seq{x_j \in D}{j \geq 1}$, bounded in $X$,
and such that $h(x_j) \rightarrow_{j\to \infty} 0$,
has a convergent subsequence in $X$.
\end{enumerate}
In the point set topology literature, Property~\PF\ corresponds to the 
property that the restriction of $h$ to any closed bounded set of $X$ is
\NEW{proper} at $0$.
If $X$ is finite dimensional, any continuous homogeneous 
map $h:X\to Y$ has Property~\PF.
When $h$ is a bounded linear map, $h$ has Property~\PF\ if, and only if,
$h$ is a semi-Fredholm linear operator with index in 
$\Z\cup\{-\infty\}$, which means that $h$ has a finite dimensional
kernel and a closed range,
see for instance~\cite[Proposition~19.1.3]{hormander}
or~\cite[Chapter IV, Theorems~5.10 and~5.11]{kato}.
In the sequel, $\id$ denotes the identity map over any set.

\begin{lemma}\label{sp-lem-ksetcontract}
If $D$ is a subset of a Banach space $(X,\norm{\cdot})$, and if $h:D\to X$
is such that $\mes[D]{h}<1$, then $\id -h$ has
Property~\PF.
\end{lemma}
\begin{proof}
Let $\seq{x_j \in D}{j \geq 1}$ be a bounded sequence 
such that $(\id - h)(x_j)\to 0$
when $j\to\infty$, and
let $S=\set{x_j}{j\geq 1}$, $T=\set{(\id - h)(x_j)}{j\geq 1}$.
Since $(\id-h)(x_j)$ converges when $j\to\infty$, 
the set $T$ is relatively compact in $X$,
and by~\eqref{sp-gmnc1}, $\mesname(T)=0$. 
Since $S$ is bounded, $\mesname(S)$ is finite.
Since $x_j = (\id-h)(x_j)+h(x_j)$, we get that
$S\subset T+h(S)$. 
Applying $\mesname$ and using~\eqref{sp-gmnc6} and~\eqref{sp-gmnc2}, 
we get
\begin{align*}
\mesname(S) & \leq \mesname(T + h(S)) \leq \mesname (T) + \mesname (h(S)) 
 = \mesname (h(S)) \leq \mes[D]{h} \mesname(S)
\end{align*}
and since 
$\mes[D]{h} < 1$ and $\mesname(S)$ is a finite nonnegative
number, it follows that $\mesname(S) = 0$.
Using Property~\eqref{sp-gmnc1}, we get that $S$ is relatively compact in $X$.
Therefore $\seq{x_j}{j \geq 1}$ has a convergent subsequence in $X$.
\end{proof}
\begin{proposition}\label{sp-prop1}
If $C$ is a cone in a Banach space $(X,\norm{\cdot})$, 
if $h:C \rightarrow C$ is homogeneous
and uniformly continuous on bounded sets of $C$,
and if either $\essp[C]{h}< 1$ or $\bonsall[C]{h}< 1$,
then $\id-h$ has Property $\PF$ on $C$.
Moreover,
when $\bonsall[C]{h}< 1$, $0$ is the unique fixed point of $h$ in $C$.
\end{proposition}
\begin{proof}
Let $\seq{x_j \in C}{j \geq 1}$ be a bounded sequence 
such that $(\id - h)(x_j)\to 0$
when $j\to\infty$.
For all $n\geq 1$, we can write:
\begin{align}
\label{sp-i-hn-telescopic}
(\id - h^n)(x_j) = \sum_{m=0}^{n-1} (h^m(x_j) - h^{m+1}(x_j))
\enspace .
\end{align}
We claim that for all $m\geq 0$,
\begin{align}
\label{sp-hm-hm+1}
h^m(x_j) - h^{m+1}(x_j) \to 0 \mrm{ when }j\to \infty
\enspace .
\end{align}
Since \eqref{sp-hm-hm+1} holds by assumption
when $m=0$, let us assume by induction that~\eqref{sp-hm-hm+1} holds
for some $m\geq 0$. Since $h$ is continuous and homogeneous,
it follows from~\eqref{sp-definormh} that the sequences
$\seq{h^m(x_j)}{j \geq 1}$ 
and $\seq{h^{m+1}(x_j)}{j \geq 1}$ 
are bounded. Since $h$ is uniformly continuous on bounded sets,
we get from~\eqref{sp-hm-hm+1} that
$h^{m+1}(x_j) - h^{m+2}(x_j) \to 0 \mrm{ when }j\to \infty$,
which shows by induction that~\eqref{sp-hm-hm+1} holds
for all $m\geq 0$. Combining~\eqref{sp-hm-hm+1} and~\eqref{sp-i-hn-telescopic},
we get that for all $n \geq 1$,
\begin{equation}\label{sp-*1}
(\id - h^n)(x_j)\to 0 \mrm{ when } j\to\infty
\enspace.
\end{equation}
Assume first that $\essp[C]{h} < 1$
so that $\mes[C]{h^n} < 1$ for some $n \geq 1$.
Then, by Lemma~\ref{sp-lem-ksetcontract},
$\id-h^n$ has Property~\PF, and we deduce
from~\eqref{sp-*1} that $\seq{x_j}{j \geq 1}$
has a convergent subsequence in $X$, so that $\id-h$
has Property~\PF.

Assume now that $\bonsall[C]{h}<1$,
so that $\norm[C]{h^n}<1$ for some $n\geq 1$.
Since
\[ 
\norm{x_j} \leq \norm{x_j-h^{n}(x_j)} +\norm{h^{n}(x_j)}
\leq \norm{x_j-h^{n}(x_j)}+ \norm[C]{h^n} \norm{x_j}\enspace,\]
we obtain, using~\eqref{sp-*1}
\[ (1-\norm[C]{h^n})\norm{x_j} \leq \norm{x_j-h^{n}(x_j)}
\to 0\mrm{ when } j\to\infty\enspace,\]
hence $x_j$ converges to $0$.
This shows that $\id -h$ has Property \PF.
Finally, if $x$ is a fixed point of $h$, we get that $h^n(x)=x$, 
hence $\norm{x}=\norm{h^n(x)}\leq \norm[C]{h^n} \norm{x}$, and since $\norm[C]{h^n}<1$,
we deduce that $x=0$.
\end{proof}

\section{Existence of eigenvectors of order-preserving maps in sup-semilattices}
\label{sp-sec-existence}
Let $C$ be a proper cone in a Banach space $X$. We say that a subset
$D$ of $X$ is a 
\NEW{sup-semilattice} (in the ordering from $C$) if, for all $x,y\in D$, there
exists $z\in D$ such that 
$x\leq z, y\leq z$ and
$z\leq w$ for every $w\in D$ for which $x\leq w$ and $y\leq w$.
Then, the element $z$ as above is unique and we shall write it $w=x\vee y$.
If $C$ is not normal, it may easily happen (see
the remark after Lemma 3.6 in~\cite{Nuss-Mallet}) that $X$ is a
sup-semilattice but that the map $(x,y)\mapsto x\vee y$ is not continuous.
The notion of \NEW{inf-semilattice} is defined dually (by reversing the order).
A \NEW{lattice} is an ordered set which is both an inf-semilattice and a sup-semilattice.

\begin{theorem}\label{sp-th-add1-new}
Let $C$ be a proper cone in a Banach space $(X,\norm{\cdot})$,
and assume that $C$ is a sup-semilattice in the partial ordering induced 
by $C$. Assume that $h:C\to C$ is continuous,
homogeneous and order-preserving in the partial ordering from $C$,
and that $\essp[C]{h}<\cspr[C]{h} $. Then there exists
$z\in C\setminus \{0\}$ such that
\[h(z)\geq \cspr[C]{h} z
\enspace .\]
If, in addition, we assume that there exists 
$\zeta\in C\setminus \{0\}$ such that
\begin{align*}
 \laba{1} &\qquad h(\zeta)\geq \cspr[C]{h} \zeta,
\qquad\text{and}\\
\laba{2} &\qquad \set{\frac{h^k(\zeta)}{\cspr[C]{h}^k}}{k\geq 0} \text{ is bounded,}
\end{align*}
then $h(x)=\cspr[C]{h} x$ for some $x\in C\setminus\{0\}$,
and, in particular, $\eigenvalspr[C]{h}=\cspr[C]{h}$.
\end{theorem}
\begin{proof}%
By replacing $h$ by $r^{-1} h$, where $r=\cspr[C]{h}$, we 
can assume that $r=\cspr[C]{h}=1$.
Since  $\essp[C]{h}<\cspr[C]{h}=1$, 
there exists $N\geq 1$ such that $\mes[C]{h^m}<1$, for all $m\geq N$.
Select a fixed $m\geq N$.
Theorem~\ref{sp-th-nus2} implies that there exists $x_m\in C\setminus\{0\}$
such that $h^m(x_m)=x_m$. 
Define $z:=x_m\vee h(x_m) \vee\cdots\vee h^{m-1} (x_m)$,
so $z\geq h^j(x_m)$ for all $j\geq 0$.
It follows that $h(z)\geq h^{j+1}(x_m)$ for all $j\geq 0$
($h$ is order-preserving), and since
$h^m(x_m)=x_m$, $h(z)\geq h^j(x_m)$ for all $j\geq 0$.
By the properties of the sup-semilattice, we must have $h(z)\geq z$.

Now assume that there exists $\zeta\in C\setminus \{0\}$ satisfying  
\laba{1} and \laba{2}.
This implies that the sequence $(h^k(\zeta))_{k\geq 0}$ is nondecreasing.
Let $S:= \set{h^{k} (\zeta)}{k\geq 0}$, and for $0\leq j\leq m-1$,
let $S_j:= \set{h^{km+j} (\zeta)}{k\geq 0}$. Notice that we have
$S_j=\{h^{j}(\zeta)\}\cup h^m(S_j)$.
Since $S_j$ is assumed to be bounded, we obtain,
using~\eqref{sp-gmnc7}, that
\[ \mesname(S_j)= \mesname( h^m(S_j)) 
\leq \mes[C]{h^m} \mesname(S_j) \enspace ;\]
and since $\mes[C]{h^m}<1$, we conclude that $\mesname(S_j)=0$,
and $\closure{S_j}$ is compact.
It follows that $S=\bigcup_{j=0}^{m-1} S_j$ has compact closure.
Because $\closure{S}$ is compact, there exists a strictly increasing 
sequence of integers $k_i$ with $h^{k_i}(\zeta)\to w$.
We claim that $h^k(\zeta)$ converges towards $w$ as $k\to\infty$.
If not, there exists a sequence $\ell_i\to\infty$ with
$\norm{h^{\ell_i}(\zeta)-w}\geq \delta>0$; and
by taking a further subsequence we can assume that
$h^{\ell_i}(\zeta)\to w'$, $w'\neq w$, as $i\to \infty$.
For any fixed $i$, we have $k_j\geq \ell_i$ for sufficiently large $j$, so
$h^{k_j}(\zeta)\geq h^{\ell_i}(\zeta)$ for all $j$
large and $w\geq h^{\ell_i}(\zeta)$.
If we now let $i$ approach infinity we see that $w\geq w'$.
By symmetry of this argument, we also obtain $w'\geq w$; so $w=w'$.
This contradicts our original assumption, so we must have 
that $h^k(\zeta)\to w$ as $k\to\infty$ and $h^{k+1}(\zeta)\to w$
as $k\to\infty$. However, the continuity of $h$ at $w$ implies that 
$h^{k+1}(\zeta)\to h(w)$, so $h(w)=w$. Note that $w\geq \zeta$, so
$w\neq 0$.
\end{proof}
\begin{remark}
We conjecture that if hypotheses are as above, but we do not assume 
the existence of $\zeta\in C\setminus \{0\}$ satisfying  \laba{1} and \laba{2}
in Theorem~\ref{sp-th-add1-new}, it is still true that $h(x)=\cspr[C]{h} x$
for some $x\in C\setminus\{0\}$.
In general, fixed point theorems which deduce the existence of fixed points
of a given map $f$ by making assumptions about the behavior of (large) 
iterates $f^m$ of $f$ have been called ``asymptotic fixed point theorems'';
see~\cite{nussbaum72,nussbaum77,nussbaum85} and references to the literature there.
Proving asymptotic fixed point theorems is sometimes surprisingly difficult,
and a number of old fundamental conjectures 
(see~\cite{nussbaum72,nussbaum77,nussbaum85}) remain open.
Our conjecture here fits into this framework, inasmuch as the 
assumption $\essp[C]{h}<\cspr[C]{h} =1$ is an assertion about the 
behavior of iterates of $h$ and not directly about $h$.
\end{remark}

The hypotheses of Theorem~\ref{sp-th-add1-new} may seem difficult
to verify. However, we shall see that by exploiting the concept of \NEW{parts}
of a cone (see Section~\ref{sp-sec-intro-hilbert} for the definition),
one can give natural assumptions which imply these
hypotheses.
It is convenient
to prove a lemma first.

\begin{lemma}\label{sp-lemma-proper}
Let $C$ be a proper cone in a Banach space $X$.
Assume that $h:C\to C$ is continuous, homogeneous
and order-preserving in the partial ordering from $C$ and
that $\essp[C]{h}<\cspr[C]{h}=1 $. Assume that $\zeta$
and $w\in C\setminus\{0\}$ are such that  $h^k(\zeta)\leq w$ for
all $k\geq 0$. Then $\set{h^k(\zeta)}{k\geq 0}$ is bounded.
\end{lemma}
\begin{proof}
We argue as on pages 46-47 in~\cite{nussbaum85}.
Let $S=\set{h^k(\zeta)}{k\geq 0}$ and assume, by way of contradiction,
that $S$ is unbounded in norm. It follows that there exists a strictly
increasing sequence of integers $\seq{k_i}{i\geq 0}$ such that
$\norm{h^{k_i}(\zeta)}\geq \norm{h^j(\zeta)}$ for $0\leq j\leq k_i$ and
$\norm{h^{k_i}(\zeta)}\to\infty$. Define 
\[z_i=\frac{h^{k_i}(\zeta)}{\norm{h^{k_i}(\zeta)}},
\qquad T=\set{z_i}{i\geq 0}
\enspace,
\]
and suppose that we can prove that $\closure{T}$ is
compact. Then, by taking a further subsequence, we can assume that
$z_i$ tends to $z$ as $i\to\infty$, where $\norm{z}=1$ and $z\in C$.
However, we have, since $\norm{h^{k_i}(\zeta)}\to\infty$,
$0\leq \frac{w}{\norm{h^{k_i}(\zeta)}}-z_i\to -z$, so $-z\in C$.
Since $C$ is proper, we have a contradiction.

Thus it suffices to prove that $\closure{T}$ is compact. 
Notice that if $n\geq 1$ is a positive integer, we can write
\[ T=\set{z_i}{i\leq n}\cup h^n\left(\left\{
\frac{h^{k_i-n}(\zeta)}{\norm{h^{k_i}(\zeta)}}\mid i>n\right\}\right)\enspace.\]
Setting $B:=\set{x\in X}{\norm{x}\leq 1}$, our assumptions imply that
$\frac{h^{k_i-n}(\zeta)}{\norm{h^{k_i}(\zeta)}}\in B\cap C$ for $i>n$,
so $T\subset \set{z_i}{i\leq n}\cup h^n(B\cap C)$.
The latter equation yields $\mesname(T)\leq \mes[C]{h^n} \mesname(B\cap C)$.
Because $\essp[C]{h}<1$, we know that $\mes[C]{h^n}\to 0$ as
$n\to\infty$, so we conclude that $\mesname(T)=0$ and $\closure{T}$ is compact.
\end{proof}

In the following corollary, recall that a \NEW{periodic point} $x$ of
a map $h$ is a point such that $h^n(x)=x$ for some positive 
integer $n$.

\begin{corollary}\label{sp-proper-fix}
Let $C$ be a proper cone in a Banach space $(X,\norm{\cdot})$ and
assume that $C$ is a sup-semilattice in the partial ordering induced by $C$.
Assume that $h:C\to C$ is continuous, homogeneous and order-preserving
in the partial ordering from $C$ and that
$\essp[C]{h}<\cspr[C]{h} $. Assume 
that there exists an integer $m$ such that
if $x\in C\setminus\{0\}$ is any periodic point of $h$,
$\set{h^j(x)}{j\geq 0}$ is contained in the union of
at most $m$ parts of $C$, i.e.\ there exist $u_1,\ldots, u_m\in C$,
$u_i$ dependent on $x$, with
\[ \set{h^j(x)}{j\geq 0}\subset \bigcup_{i=1}^{m} C_{u_i}\enspace.\]
Then there exists $y\in C\setminus\{0\}$ with $h(y)=r y$, $r=\cspr[C]{h}$.
\end{corollary}
\begin{proof}
By replacing $h$ by $r^{-1} h$, where $r=\cspr[C]{h}$, we 
can assume that $r=\cspr[C]{h}=1$.
Select a prime integer $p$ such that $\mesname(h^p)<1$ and $p> m$,
where $m$ is as in the statement of the corollary.
Theorem~\ref{sp-th-nus2} %
implies that there exists $x\in C\setminus\{0\}$
such that $h^p(x)=x$. We claim that $h^j(x)$ is comparable to $x$ for all
$j\geq 0$.
Since $p>m$, there exists $j$ with $x\sim h^j(x)$ and $0<j<p$.
Since $h$ is order-preserving and homogeneous, it follows that
$x\sim h^{kj}(x)$ for all integers $k\geq 1$. Since $p$ is
prime, there exists $k\geq 1$ such that $kj\equiv 1 \mod p$, so
$x\sim h(x)$, and applying $h$ repeatedly we see that $x\sim h^s(x)$
for all $s\geq 0$.

Since $x\sim h^j(x)$ for $0\leq j\leq p-1$, there exists
$\beta>0$ such that $h^j(x)\leq \beta h^k(x)$ for
$0\leq j\leq p-1$ and $0\leq k\leq p-1$.
If we write
$\zeta=x\vee h(x)\vee\cdots\vee h^{p-1}(x)$, it follows
that $\zeta\leq \beta h^k(x)$ for $0\leq k\leq p-1$.
A simple induction on $s$ implies that
$h^s(\zeta)\leq \beta h^k(x)$ for $0\leq k\leq p-1$ and all $s\geq 0$.
The latter equation implies that $h^s(\zeta)\leq \beta \zeta$
for all $s\geq 0$. 
So Lemma~\ref{sp-lemma-proper} implies that $\set{h^s(\zeta)}{s\geq 0}$ 
is bounded. The same argument as in Theorem~\ref{sp-th-add1-new} shows that
$h(\zeta)\geq \zeta$; and $\zeta\in C\setminus\{0\}$ because $\zeta\geq x$.
Theorem~\ref{sp-th-add1-new}  now implies that there exists 
$y\in C\setminus\{0\}$ with $h(y)=y$.
\end{proof}

\begin{remark}
Let hypotheses be as in Corollary~\ref{sp-proper-fix}, but do not
assume the existence of an integer $m$ as in the statement of 
Corollary~\ref{sp-proper-fix}. Instead assume that there exists an integer
$\mu\geq 1$ and an integer $N\geq 1$ and points $v_1,\ldots, v_\mu$
in $C$ such that 
\[h^N(C)\subset \bigcup_{i=1}^{\mu} C_{v_i}\enspace . \]
Then it is easy to see that the hypotheses of Corollary~\ref{sp-proper-fix}
are satisfied with $\mu=m$.
A very special case, which is sometimes assumed in the literature
is to assume that $\Cint$, the interior of $C$, is nonempty
and $h^N(C)\subset \{0\}\cup \Cint$ for some integer $N$.
\end{remark}

\section{The spectral radius of order-preserving homogeneous self-maps of a Banach space}\label{sp-sec-X}
We now consider the situation in which $f$ is defined on the whole Banach
space $X$, and determine the spectral radius of $f$ by considering
the restriction of the map to the cones $C$ and $-C$. Recall 
our convention to write $\cspr{h}$ for $\cspr[X]{h}$, etc.

\begin{lemma} \label{sp-lem-add1}
Let $C$ be a reproducing normal cone in a Banach space
$(X,\norm{\cdot})$, and let $h:X\to X$ be a map which is 
continuous, positively homogeneous, and order-preserving with respect to 
$C$. Then, $h(C)\subset C$, $h(-C)\subset -C$ and
\begin{equation}\label{sp-lem-add11}
\bonsall{h}=\max (\bonsall[C]{h}, \bonsall[-C]{h})
=\max (\cspr[C]{h}, \cspr[-C]{h})=\cspr{h}\enspace .
\end{equation}
\end{lemma}
\begin{proof}
Since $h$ is continuous and positively homogeneous, $h(0)=0$. Since $h$ is order
preserving with respect to $C$, we deduce that $h(C)\subset C$ and
$h(-C)\subset -C$, so that $\bonsall[C]{h},\;\bonsall[-C]{h},\;
\cspr[C]{h}$ and $\cspr[-C]{h}$ are well defined.
Moreover, $h$ is also order-preserving with respect to $-C$.
It follows from Theorem~\ref{sp-prop-nus1} that
 $\bonsall[C]{h}=\cspr[C]{h}$ and $ \bonsall[-C]{h}=\cspr[-C]{h}$,
which shows the central equality in~\eqref{sp-lem-add11}.
By definition, 
$\bonsall{h}\geq \max (\bonsall[C]{h}, \bonsall[-C]{h})$ and
$\cspr{h}\geq \max (\cspr[C]{h}, \cspr[-C]{h})$.
It remains to show the reverse inequalities.

Let $x\in X$. Since $C$ is reproducing, there exist $x^+$ and $x^-\in C$
such that $x=x^+-x^-$. Then, 
$-x^-\leq x\leq x^+$. %
 Since $h$ is order-preserving, we get
$h^n(-x^-)\leq h^n(x)\leq h^n(x^+)$ for all $n\in \N$, hence
$0\leq h^n(x)-h^n(-x^-)\leq h^n(x^+)-h^n(-x^-)$.
Since $C$ is normal, it follows that
$\norm{h^n(x)-h^n(-x^-)}\leq M\norm{h^n(x^+)-h^n(-x^-)}$
for some positive constant $M$,
hence 
\begin{equation}\norm{h^n(x)}\leq M\norm{h^n(x^+)}+ (M+1) \norm{h^n(-x^-)}\enspace .
\label{sp-lem-add12}\end{equation}
Taking the power $1/n$ of this inequality and passing to 
the limit when $n$ goes to infinity, we obtain:
$\mu(x)\leq \max(\mu(x^+), \mu(x^-))\leq \max (\cspr[C]{h},\cspr[-C]{h})$,
where $\mu$ is defined as in~\eqref{sp-eq31-cspr}.
Therefore, $\cspr{h} \leq \max (\cspr[C]{h},\cspr[-C]{h})$.

Using~\eqref{sp-lem-add12} again, we deduce that 
\begin{align}
 \norm{h^n(x)}& \leq M\norm[C]{h^n} \norm{x^+}+ (M+1) \norm[-C]{h^n} \norm{x^-}\nonumber \\
&\leq (M+1) \max(\norm[C]{h^n},\norm[-C]{h^n} ) (\norm{x^+}+\norm{x^-})\enspace .
\label{sp-lem-add13}
\end{align}
Since $C$ is a reproducing proper cone, the quantity
\[ \trn x\trn:=\inf\set{\norm{x^+}+\norm{x^-}}{ x^+, x^-\in C, \; x=x^+-x^-},
\quad x\in X\enspace ,\]
defines a norm on $X$ which is equivalent to the initial norm $\norm{\cdot}$
(see~\cite{schaefer2}).
Hence, there exists $M_0>0$ such that $\trn x\trn \leq M_0\norm{x}$ for 
all $x\in X$. Since~\eqref{sp-lem-add13} holds for all $x^+,x^-\in C$ 
such that $x=x^+-x^-$, we deduce that
$\norm{h^n(x)}\leq (M+1) \max(\norm[C]{h^n},\norm[-C]{h^n} ) \trn x\trn 
\leq (M+1) M_0\max(\norm[C]{h^n},\norm[-C]{h^n} ) \norm{x}$, and since this
holds for all $x\in X$, we obtain $\norm{h^n}\leq (M+1)
 M_0\max(\norm[C]{h^n},\norm[-C]{h^n} )$. It follows that
$\bonsall{h}\leq \max (\bonsall[C]{h}, \bonsall[-C]{h})$.
\end{proof}
The following proposition identifies
a situation in which the uniqueness
of the fixed point of a map
implies that every orbit converges to this fixed point 
with a geometric rate.
\begin{theorem}[Uniqueness implies contraction]\label{sp-prop-add1}
Let $C$ be a reproducing normal cone in a Banach space $(X,\norm{\cdot})$,
and let $h:X\to X$ be a map that is continuous, positively homogeneous,
and order-preserving with respect to $C$. Then, $\bonsall{h}=\cspr{h}$ 
and 
\begin{align}\label{sp-prop-add11}
\mesname(h)< \cspr{h} \implies 
\left\{ \begin{array}{l}
\cspr{h}= \eigenvalspr{h}\mrm{ and there exists }\\
 x\in (C\cup -C)\setminus\{0\},\; h(x)= \cspr{h} x
\enspace .
\end{array}\right. 
\end{align}
In particular, if we assume that
$\cspr{h}\leq 1$ and that $\mesname(h)<1$, we have:
\begin{align}\label{sp-prop-add12}
h  \mrm{ has a unique fixed point, i.e.\ } (h(x)=x\Rightarrow x=0)
\;
\iff \cspr{h} <1 \enspace .
\end{align}
\end{theorem}
\begin{proof}
It follows from Lemma~\ref{sp-lem-add1} that $\bonsall{h}=\cspr{h}$.
Assume first that $\mesname(h)< \cspr{h}$. 
By Lemma~\ref{sp-lem-add1}, we deduce that either $\cspr{h}=
\cspr[C]{h}$ or $\cspr{h}=\cspr[-C]{h}$. Consider for instance the
case in which $\cspr{h}= \cspr[C]{h}$. Using
$\mes[C]{h}\leq \mesname(h)$ and Theorem~\ref{sp-th-nus2}
 for $m=1$, we deduce that
$\cspr[C]{h}= \eigenvalspr[C]{h}$ and that there exists 
$x\in C\setminus\{0\}$ such that $h(x)= \cspr[C]{h} x$,
which shows~\eqref{sp-prop-add11}.
The case where $\cspr{h}=\cspr[-C]{h}$ is similar.
Since $\eigenvalspr{h}\leq \cspr{h}$ and there exists $x\in 
X\setminus\{0\}$ such that $h(x)= \cspr{h} x$, we obtain the equality
$\eigenvalspr{h}=\cspr{h}$.

Assume now that $\cspr{h}\leq 1$ and that $\mesname(h)<1$.
Then, if  $\cspr{h} \not <1$, we get that $\cspr{h} =1$, hence,
by~\eqref{sp-prop-add11}, there exists  $x\in (C\cup -C)\setminus\{0\}$
such that $h(x)= \cspr{h} x=x$, that is a non zero fixed point
of $h$, which shows the $\Rightarrow$ implication in~\eqref{sp-prop-add12}.
The converse implication follows from $\eigenvalspr{h}\leq \cspr{h}$.
\end{proof}
\begin{proposition}\label{sp-th-add1}
Let $C$ be a reproducing normal cone in a Banach space $(X,\norm{\cdot})$,
such that $X$ is a lattice for the order defined by $C$.
Let $h:X\to X$ be a map which is continuous,
positively homogeneous, order-preserving with respect to $C$,
and such that for all $x\in X$, the orbit $\{h^k(x)\}_{k\in\N}$
is bounded. 
Then, $\bonsall{h}=\cspr{h}\leq 1$.
Moreover, if $\essp{h}<1$ and $\cspr{h}=1$, then
there exists $x\in (C\cup -C)\setminus \{0\}$
such that $h(x)=x$, and in particular $\cspr{h}=\eigenvalspr{h}$.
Finally, if $\essp{h}<1$, then~\eqref{sp-prop-add12} holds.
\end{proposition}
\begin{proof}
Since all the orbits of $h$ are bounded,
$\cspr{h}\leq 1$.
By Lemma~\ref{sp-lem-add1}, $h(C)\subset C$, 
$h(-C)\subset -C$, and
$\bonsall{h}=\cspr{h}=\max(\cspr[C]{h},\cspr[-C]{h})$. 
Assume that $\cspr{h}=1$ and that $\essp{h}<1$.
Then, $\cspr[C]{h}=1$ or $\cspr[-C]{h}=1$. 
Assume first that $\cspr[C]{h}=1$.
Since $X$ is a lattice, $C$ is a sup-semilattice
for the order defined by $C$ and since $\essp[C]{h}\leq \essp{h}<1$,
Theorem~\ref{sp-th-add1-new} shows that there exists $x\in C\setminus\{0\}$
such that $h(x)=x$. When $\cspr[-C]{h}=1$,
a symmetrical argument shows that there exists $x\in -C\setminus\{0\}$
such that $h(x)=x$. Since $\eigenvalspr{h}\leq \cspr{h}$,
this implies $\eigenvalspr{h}=\cspr{h}$.
If now $\essp{h}<1$ only, the same arguments as in the proof of
Theorem~\ref{sp-prop-add1} show that~\eqref{sp-prop-add12} holds.
\end{proof}
\begin{remark}
When $C$ is a reproducing normal cone in a Banach space $(X,\norm{\cdot})$,
the condition that $C$ is a sup-semilattice for the order defined by $C$ 
together with the condition
that $-C$ is a sup-semilattice for the order defined by $-C$
is equivalent to the condition assumed in Proposition~\ref{sp-th-add1}
that $X$ is a lattice for the order defined by $C$.
When $X$ is an AM-space with unit, i.e.,
when $X$ is the space of continuous functions
on a compact space, equipped with the sup-norm~\cite{aliprantis},
and when $C=X^+$, then $C$ and $X$ satisfy
the assumptions of Proposition~\ref{sp-th-add1}.
\end{remark}

We also have:
\begin{lemma} \label{sp-lem-add2}
Let $C$ be a normal cone in a Banach space $(X,\norm{\cdot})$, and let
$g$ and $h:C\to C$ be two maps which are continuous and homogeneous.
Assume that $g$ is order-preserving and that $g\leq h$, then
\begin{equation} \label{sp-lem-add21}
\bonsall[C]{g}= \cspr[C]{g}\leq \cspr[C]{h}\leq \bonsall[C]{h}
\enspace .
\end{equation}
\end{lemma}
\begin{proof}
The equality in~\eqref{sp-lem-add21} follows from Theorem~\ref{sp-prop-nus1} 
and the last inequality holds in general.
Let $x\in C$. We show by induction on $n$
that $0 \leq g^n(x)\leq h^n(x)$ holds for all $n\geq 1$. 
Since $x\in C$, $g(C)\subset C$,
and $g\leq h$, we get $0\leq g(x)\leq h(x)$. Assume now
that $0 \leq g^n(x)\leq h^n(x)$ for some $n\geq 1$.
Then, applying $g$, which is order-preserving,
we get $g(0)\leq g^{n+1}(x)\leq g(h^n(x))$. 
Since $g\leq h$, we get $g^{n+1}(x)\leq h^{n+1}(x)$,
and since $0=g(0)$, we deduce that $0\leq g^{n+1}(x)$,
which concludes the induction. 
Since $C$ is normal, we deduce from $0\leq g^n(x)\leq h^n(x)$
that $\norm{g^n(x)}\leq M\norm{h^n(x)}$
for some positive constant $M$.
Since this holds for all $x\in C$, it follows that 
$\cspr[C]{g}\leq \cspr[C]{h}$.
\end{proof}

We say that a map $f:X\to X$ is \NEW{convex} (for the order of $C$) if
$f(t x+(1-t)y)\leq tf(x)+(1-t) f(y)$ for all $0\leq t\leq 1$, $x,y\in X$.
\begin{corollary}\label{sp-cor-add2}
Let $C$ be a reproducing normal cone in a Banach space 
$(X,\norm{\cdot})$, and let
$f:X\to X$ be a map which is continuous, positively homogeneous,
order-preserving.
Assume also that for all $x\in C$, $(f(x)+f(-x))$ is an element
of $C$.
Then, $\bonsall{f} =\bonsall[C]{f}=
\cspr{f} =\cspr[C]{f}$.
\end{corollary}
\begin{proof}
Consider the map $f^-:X\to X$ defined by 
$f^-(x)=-f(-x)$. We have $(f^-)^n(x)=-f^n(-x)$ for all $x\in X$, 
hence $\norm[C]{(f^-)^n}=\norm[-C]{f^n}$ for all $n\in\N$.
It follows that $\cspr[C]{f^-}= \cspr[-C]{f}$ and
$\bonsall[C]{f^-}= \bonsall[-C]{f}$.
We get that 
for all $x\in X$, $0
\leq  f(x)+ f(-x)$, hence 
$-f(-x)\leq f(x)$, which shows that $f^-\leq f$.
Since $f^-$ is order-preserving (with respect to $C$), we deduce from
Lemma~\ref{sp-lem-add2} that $\cspr[C]{f^-}\leq \cspr[C]{f}$
and $\bonsall[C]{f^-}\leq \bonsall[C]{f}$.
With Lemma~\ref{sp-lem-add1}, this yields the assertion of the corollary.
\end{proof}
Note
that the condition that $(f(x)+f(-x))$ belongs to $C$ used
in the corollary holds in particular if $f$ is convex,
meaning that $\frac{1}{2}(f(x)+f(y)) \geq f(\frac{1}{2}(x+y))$
holds for all $x,y \in X$.

\section{Spectral radius and Collatz-Wielandt number}
\label{sp-sec-CW}
Let $C$ be a cone in a Banach space $(X, \norm{\cdot})$,
assume $C$ has nonempty interior $\Cint$, 
and let $h:C\to C$ be a homogeneous map. We define
the \NEW{Collatz-Wielandt number}:
\begin{align}
\cw[C]{h} & =\inf\, \set{ \lambda >0}{\exists x \in \Cint, \; h(x) \leq
 \lambda x } \enspace .\label{e-def-cw}
\end{align}
When the choice of the cone $C$ will be obvious,
and in particular, when $C=X$, the cone $C$
will be omitted in the previous notation.

The next proposition allows us to rewrite the Collatz-Wielandt
number in a perhaps more familiar way.
\begin{proposition}
If $C$ is a proper cone with non-empty interior, then, 
\begin{align}
\cw[C]{h} &= \inf_{x\in\Cint}\sup_{\psi\in C^*\setminus\{0\}}
\frac{\psi(h(x))}{\psi(x)}
 \enspace,\label{sp-e-cwf}\\
&= \inf_{x\in \Cint} 
\sup_{\psi\in \operatorname{extr}C^*}
\frac{\psi(h(x))}{\psi(x)}
 \enspace,\label{sp-e-cwf2}
\end{align}
where $\operatorname{extr}C^*$ denotes the set
of non-zero elements of extreme rays of the cone $C^*$.
\end{proposition}
\begin{proof}
Indeed, if $C$ is a proper cone, the Hahn-Banach theorem implies that $h(x)\leq \lambda x$ if and only if $\psi(h(x)) \leq \psi(\lambda x)$ for all $\psi$ in $C^*$.
Observe that if $\psi\in C^*$
is non-zero, then $\psi$ is non-zero at every point
of the interior of $C$. Then, the expression~\eqref{sp-e-cwf} is
readily derived from~\eqref{e-def-cw}.
Now, let $x$ be an arbitrary vector in the interior of $C$
and take $\epsilon>0$ so that the ball
$B(x,\epsilon)$ is included in $C$. Then,
$y\leq \epsilon^{-1}x$ holds for all $y\in B(0,1)$,
and so, $\psi(y)\leq \epsilon^{-1} \psi(x) $. 
Since the same inequality applies to $-y$, we deduce
that $|\psi(y)|\leq \epsilon^{-1}\psi(x)$, showing that
the dual norm of $\psi$ satisfies $\|\psi\|^*\leq \epsilon^{-1}\psi(x)$.
It follows that $\Sigma(x)=\{\psi \in C^*\mid \psi(x)=1\}$ 
is bounded in the space $(X^*,\|\cdot\|^*)$, and since $\Sigma(x)$
is trivially closed in the weak-star topology,
it is compact in this topology. Then, the
map $\Sigma(x) \to \R, \; \psi \mapsto \psi(h(x))$,
which is continuous in the same topology,
achieves its maximum
at one of the extreme points of $\Sigma(x)$. Since the extreme rays
of $C^*$ are precisely the rays generated by the extreme points of $\Sigma(x)$, 
we get~\eqref{sp-e-cwf2}.
\end{proof}
The expression~\eqref{sp-e-cwf2} extends
the classical Collatz-Wielandt function
which arises in Wielandt's approach to the finite
dimensional Perron-Frobenius theorem. The latter
concerns the special case in which $C$ is the standard positive
cone $\R_+^n$ in $\R^n$:
\[
\cw[\R_+^n]{h} = 
\inf_{x\in\operatorname{int}\R_+^n}\sup_{1\leq i\leq n}
\frac{(h(x))_i}{x_i} \enspace.
\]

We now examine the properties of the Collatz-Wielandt number
and relate it with the different notions of spectral radius. 
We start
with an obvious observation.
\begin{lemma} \label{sp-lem-add4}
Let $C$ be a proper cone in a Banach space $(X,\norm{\cdot})$, with nonempty
interior, and let
$g$ and $h:C\to C$ be two maps which are homogeneous.
If $g\leq h$, then
\begin{equation} \label{sp-lem-add41}
\cw[C]{g}\leq \cw[C]{h}\enspace .
\end{equation}
\end{lemma}

We also have:
\begin{lemma} \label{sp-lem-add5}
Let $C$ be a proper cone in a Banach space $(X, \norm{\cdot})$, 
with nonempty interior,
and let  $h:C\to C$ be a continuous, homogeneous
and order-preserving map.
Then, for all $n\geq 1$, we have
\begin{equation}\label{sp-lem-add52}
\cczech[C]{h}\leq ( \cw[C]{h^n})^{1/n}
\leq\cw[C]{h} \enspace . \end{equation}
If in addition $C$ is normal, we also have, for all $n\geq 1$,
\begin{equation}\label{sp-lem-add51}
\cspr[C]{h}=\bonsall[C]{h}\leq ( \cw[C]{h^n})^{1/n}
 \enspace . \end{equation}
\end{lemma}
\begin{proof}
Let $\lambda>0$ and $x\in \Cint$ be such that $h(x)\leq \lambda x$.
Since $h$ is homogeneous 
and order-preserving, we deduce that $h^n(x)\leq \lambda^n x$ for all
$n\geq 1$.
Hence, $\lambda^n\geq \cw[C]{h^n}$, and since this holds for
all $\lambda>0$ and $x\in \Cint$ such that $h(x)\leq \lambda x$,
we get that $( \cw[C]{h^n})^{1/n} \leq\cw[C]{h}$
and the second inequality of~\eqref{sp-lem-add52} is proved.

We now assume that $\cczech[C]{h}>0$, since otherwise the first inequality
of~\eqref{sp-lem-add52} is trivial.
Let $n,m\geq 1$,  $\lambda,\mu>0$, $x\in  \Cint$ and 
$y\in C\setminus \{0\}$ be such that $h^n(x)\leq\lambda^n x$ 
and $h^m(y)= \mu^m y$. Then, $h^{nm}(y)= \mu^{nm} y$ and,
since $h$ is homogeneous and order preserving,
$h^{nm}(x)\leq\lambda^{nm} x$.
Since $x\in \Cint$ and $y\in C\setminus \{0\}$, we get that
$b:=\M(y/x)\in (0,+\infty)$. Since $y\leq b x$, using again the fact
that $h$ is homogeneous
and order-preserving, it follows that
$\mu^{nm} y=h^{nm}(y)\leq b h^{nm}(x)\leq b \lambda^{nm} x$.
Hence, $\M(y/x)\leq (\mu^{-1}\lambda)^{nm} b=(\mu^{-1}\lambda)^{nm}\M(y/x)$,
and 
since $\M(y/x)\in (0,+\infty)$,
we deduce that $\mu\leq \lambda$. Since this holds for all
$\lambda>0$ and $\mu>0$ as above, we get
\[(\eigenvalspr[C]{h^m})^{1/m}\leq (\cw[C]{h^n})^{1/n}
\enspace .\]
Since this inequality holds for all $n,m\geq 1$, we obtain the
first inequality of~\eqref{sp-lem-add52}.

Assume now that $C$ is normal.
By Theorem~\ref{sp-prop-nus1}, $\bonsall[C]{h}=\cspr[C]{h}$.
Since for all $m\geq 1$,
$h^m$ is continuous, homogeneous and order-preserving, and
$(\bonsall[C]{h})^m=\bonsall[C]{h^m}$ it is sufficient to prove the
inequality of~\eqref{sp-lem-add51} for $m=n=1$.
Let $\lambda>0$, let $x\in \Cint$ be such that $h(x)\leq \lambda x$
and let $y\in C$. Then, $b:=\M(y/x)\in (0,+\infty)$ and $y\leq b x$. 
As above, it follows that $h^k(y)\leq b h^k(x)\leq b \lambda^k x$,
for all $k\geq 1$.
Since $C$ is normal, there exists $M>0$ (independent of
$y\in C$ and $k$) such that $\norm{h^k(y)}\leq M b \lambda^k \norm{x}$, 
for all $k\geq 1$.
Hence, $\mu(y)\leq \lambda$ holds for all $y\in C$,
and so $\cspr[C]{h}\leq \lambda$. Since this holds 
for all $\lambda>0$ and $x\in \Cint$ such that $h(x)\leq \lambda x$,
this concludes the proof of~\eqref{sp-lem-add51}.
\end{proof}
The following theorem extends the characterization
of the Perron root as the value of the Collatz-Wielandt
function, which arises in finite dimensional (linear) Perron-Frobenius
theory.
\begin{theorem}\label{sp-supeigenth0}
Let $C$ be a normal cone in a Banach space $(X,\norm{\cdot})$, with
nonempty interior.
Let $h:C\to C$ be a continuous, homogeneous, and order-preserving
map. 
Consider the following conditions:
\begin{center}
\begin{tabular}{cl}
(i)& $h$ is uniformly continuous on bounded sets of $C$,\\
(ii)&$\mes[C]{h}< \cw[C]{h}$,\\
(iii)&$\essp[C]{h}<\cw[C]{h}$.
\end{tabular}
\end{center}
If (i) or~(ii) holds, then
\[\cspr[C]{h}=
\cw[C]{h} \enspace.\]
If ((i) and (iii)) or (ii) holds, then
\[
\eigenvalspr[C]{h}= \cspr[C]{h} \enspace .
\]
and there exists an element $x\in C\setminus\{0\}$
such that
\[
h(x)=\cspr[C]{h} x \enspace .
\]
\end{theorem}
To prove this result, we need the following lemmas.

\begin{lemma}\label{sp-lemma-gen}
Let $C$ be a proper cone in a Banach space $(X,\norm{\cdot})$.
For all $u,x,y\in C\setminus\{0\}$ such that $\M(x/u)<+\infty$ and
$\M(y/x)<+\infty$, we have
\[ \M(y+u / x+ u)\leq 
\frac{(\M(y/x)\vee 1 ) \M(x/u)+1}{\M(x/u)+1}\enspace.\]
\end{lemma}
\begin{proof}
Let $u,x,y$ be as in the lemma, and denote $b=\M(y/x)\vee 1$. Then, $b \geq 1$
and $y\leq b x$.
We need to find $b_1\geq 0$ such that
$y+u \leq b_1(x+ u)$. Since $y\leq b x$, the latter
inequality is satisfied if $b x+u \leq b_1(x+ u)$,
or equivalently if
\[ (b-b_1) x\leq (b_1-1) u\enspace.\]
This inequality holds when $1\leq b_1\leq b$ and
\[ (b-b_1) \M(x/u) \leq (b_1-1)\enspace. \]
Since $b\geq 1$, we can take
$b_1=(b \M(x/u)+1)/(\M(x/u)+1)$,
which yields the inequality of the lemma.
\end{proof}
Results closely related to the following lemma can be found in
Lemma~2.1, p.~45, and Theorem~2.6, p.~59,
of~\cite{nussbaum88}. 
See, also, Lemma~3.9, p.~216, in~\cite{Nus07}.
\begin{lemma}\label{sp-lem-eps1}  
Let $C$ be a proper cone in a Banach space $(X,\norm{\cdot})$, with
nonempty interior $\Cint$ and let $\hil$ denote Hilbert's projective metric.
Let $q:C\setminus\{0\} \to (0,+\infty)$ be a homogeneous map
preserving the order of $C$, 
and let $\Sigma:=\set{x\in C\setminus\{0\}}{q(x)=1}$. 
Given $u\in \Cint$, define the maps $\Phi_u: C\setminus\{0\} \to\Cint$ and
$\Psi_u:C\setminus\{0\}\to\Cint\cap\Sigma$ by
\[
\Phi_u(x) = x + q(x) u\quad \mrm{ and }\quad \Psi_u(x)=
 \frac{\Phi_u(x)}{q(\Phi_u(x))} \enspace .
\]
Then, for all $v\in C\setminus\{0\}$ and  $R>0$,
there exists a constant $c=c_{u,v,R}$ (depending on $u$, $v$
 and $R$) such that $0\leq c<1$ and
\begin{equation}\label{sp-lem-eps2}
\hil(\Psi_u(x), \Psi_u(y)) =\hil(\Phi_u(x), \Phi_u(y))
\leq c \; \hil(x,y) \qquad \forall x,y
\in B_R(v)\enspace,
\end{equation}
where $B_R(v):=\set{x\in C_v}{\hil(x,v)\leq R}$.
In particular, $\Psi_u|_{C_v\cap\Sigma}$ is nonexpansive and it is 
a contraction mapping 
on all bounded sets of the metric space $(C_v\cap\Sigma,d)$.

Moreover, if $C$ is normal and there exists $\gamma >0$ such that 
\begin{equation}\label{sp-qcoerc}
\norm{x}\leq \gamma q(x)\quad\forall x\in C\setminus\{0\}\enspace,\end{equation}
the image of $C\setminus\{0\}$ by $\Psi_u$ is a bounded subset of 
$(\Cint\cap\Sigma,d)$.
\end{lemma}

\begin{proof}
Since $q(x)>0$ for all $x\in C\setminus\{0\}$,
it is clear that $\Phi_u$ sends $C\setminus\{0\}$ into
$\Cint$ and $\Psi_u$ sends $C\setminus\{0\}$ into $\Cint\cap\Sigma$.
Moreover, if~\eqref{sp-lem-eps2} is shown for all $R>0$, then
$\hil(\Psi_u(x), \Psi_u(y)) \leq \hil(x,y)$ for all $x,y\in C_v$,
which implies that $\Psi_u|_{C_v\cap\Sigma}$ is nonexpansive in $\hil$.

Let $v\in C\setminus\{0\}$ and  $R>0$, and let us show~\eqref{sp-lem-eps2}.
Since the Hilbert's projective metric only depends on the
lines generated by two vectors, 
we have $\hil(\Psi_u(x), \Psi_u(y))=\hil(\Phi_u(x), \Phi_u(y))$
for all $x,y\in C\setminus\{0\}$, which shows the equality 
in~\eqref{sp-lem-eps2}.
Moreover, by homogeneity of $q$,  
we get that for all $x,w\in C\setminus\{0\}$
$\Psi_u(x)=\Psi_u(x')$ and $\hil(x,w)=\hil(x',w)$
where $x'=\frac{x}{q(x)}\in\Sigma$,
hence it is sufficient to show the 
inequality in~\eqref{sp-lem-eps2} when $x,y\in  B_R(v)\cap \Sigma$. 
We can also assume without loss of generality that $v\in \Sigma$.
In addition, since $\hil(x,y)=\log \M(y/x)+\log\M(x/y)$ for all 
$x,y,w\in C\setminus\{0\}$ such that $x,y\in C_w$, it is
sufficient to show that 
\begin{equation}
\M(\Phi_u(y)/ \Phi_u(x)) \leq \left[\M(y/x)\right]^c 
\qquad \forall x,y \in B_R(v) \cap \Sigma\enspace.
\end{equation}

Let $x,y\in B_R(v)\cap \Sigma$. 
Since $u\in\Cint$ and $v\in C\setminus\{0\}$,
we deduce that $M_0:=\M(v/u)\in (0,+\infty)$.
{From}~\eqref{sp-mleqd}, we deduce
that $\M(x/v)\leq e^{\hil(x,v)}\leq e^R$ (since $q(x)=q(v)=1$
and $x\in C_v$), hence
$\M(x/u)\leq \M(x/v) \M(v/u)\leq e^R M_0$.
Using this inequality together with~\eqref{sp-mleqd} and 
Lemma~\ref{sp-lemma-gen}, we get that 
\[ \M(\Phi_u(y)/\Phi_u(x))\leq \mu+ (1-\mu) \M(y/x)\quad
\mrm{where}\quad \mu=\frac{1}{e^R M_0+1}\enspace.\]
Hence, it is sufficient to show that there exists $0\leq c<1$,
independent of $x,y\in B_R(v)\cap \Sigma$, such that
\[ \mu+ (1-\mu) \M(y/x)\leq \left[\M(y/x)\right]^c \enspace.\]
Since $\hil(x,y)\leq \hil(x,v)+\hil(v,y)\leq 2R$, it follows from~\eqref{sp-mleqd} that
$ 1\leq \M(y/x)\leq e^{2R}$.
Hence, it is sufficient to show that 
\begin{equation}\label{sp-elem-conv}
 \mu+ (1-\mu) \beta \leq \beta^c \quad \forall\beta\in [1,e^{2R}]
\enspace.\end{equation}
Let $0\leq c<1$ and define
$\varphi(\beta)=\mu+(1-\mu)\beta-\beta^c$ for $\beta>0$. We get that
$\varphi(1)=0$ and since $\varphi$ is convex, the 
inequality~\eqref{sp-elem-conv} holds if, and only if,
$\varphi(e^{2R})\leq 0$, which is satisfied when
\[ c=\frac{\log(\mu+(1-\mu) e^{2R})}{2R}\enspace .\]
Since $0<\mu<1$ and $R>0$, we get that $0<c<1$ and~\eqref{sp-lem-eps2} is proved.

Assume now that~\eqref{sp-qcoerc} holds.
Let  $x\in C\setminus\{0\}$.
We get that $q(x) u\leq \Phi_u(x)$, hence $\m(\Phi_u(x)/u)\geq q(x)$.
Since $C$ is normal and $u\in\Cint$, the norms $\norm{\cdot}$ and 
$\norm[u]{\cdot}$ are equivalent in $X$ (see Proposition~\ref{sp-prop-normequiv}).
In particular, there exists a constant $\beta>0$ such that
$\norm[u]{x}\leq \beta \norm{x}$ for all $x\in X$.
{From}~\eqref{sp-qcoerc}, we deduce that
$\norm[u]{x}\leq \beta \norm{x}\leq \beta \gamma q(x)$.
It follows that 
$\Phi_u(x)\leq (\beta \gamma +1) q(x) u$, or equivalently
$\M(\Phi_u(x)/u)\leq (\beta \gamma +1) q(x) $. 
With $\m(\Phi_u(x)/u)\geq q(x)$, this yields 
$\hil(\Psi_u(x),u)=\hil(\Phi_u(x),u)\leq \log (\beta \gamma +1)$,
which shows that the image of $\Psi_u$ is bounded
in $(\Cint\cap\Sigma,d)$.
\end{proof}

\begin{remark}
The conclusion of Lemma~\ref{sp-lem-eps1} remains
valid, with the same constant $c$, if 
the map $\Phi_u$ (or $\Psi_u$) is replaced by the
map $C\setminus\{0\}\to \Cint,\; x\mapsto x+u$,
$\Sigma$ is replaced by $C\setminus\{0\}$,
and Hilbert's projective
metric $\hil$ is replaced by Thompson's metric $\tho$.
\end{remark}

\begin{proof}[Proof of Theorem~\ref{sp-supeigenth0}]
We shall assume that $h$ is a non identically zero map on $C$,
since otherwise all the assertions of the theorem are trivially true.
By Theorem~\ref{sp-prop-nus1}, $\cspr[C]{h}=\bonsall[C]{h}$,
so we shall prove that
$\bonsall[C]{h}=\cw[C]{h}$.
To do so, we shall construct an approximation
of $h$ which has an eigenvector in the interior of $C$.

Let $u\in\Cint$. Since $C$ is a normal cone, 
$\norm[u]{\cdot}$ is a norm equivalent to $\norm{\cdot}$ on $X$
(see Proposition~\ref{sp-prop-normequiv}), 
and since $\norm[u]{\cdot}$ is order-preserving on $C$, 
the restriction $q$ of $\norm[u]{\cdot}$ to $C\setminus\{0\}$
satisfies all the conditions of Lemma~\ref{sp-lem-eps1}.
Since $h$ is homogeneous and order-preserving,
$h$ sends $\Cint=C_u$ into $C_{h(u)}$
and $h$ is nonexpansive in Hilbert's projective metric $\hil$
(see Lemma~\ref{sp-lemma-nonexpan}).
Moreover, since $h$ is nonzero and continuous, 
$h(\Cint)$ cannot be $\{0\}$, hence $h(u)\neq 0$.
Let $s>0$, let $\Sigma$, $\Phi_{su}$ and $\Psi_{su}$ be defined as 
in Lemma~\ref{sp-lem-eps1}, 
denote $h_s=\Phi_{su}\circ h$,
$g_s=\Psi_{su}\circ h|_{\Cint\cap \Sigma}$ and $v=h(u)$. 
By the properties of $h$ and
Lemma~\ref{sp-lem-eps1}, $g_s$ is a self-map of $\Cint\cap \Sigma$,
it is nonexpansive in $\hil$ and its image is bounded for the metric $\hil$.
Moreover, since $h$ is nonexpansive in $\hil$, the image by $h$ of
any bounded set of $(\Cint\cap \Sigma,d)$ is included in some
$B_R(v)$ with $R>0$, hence
$g_s$ is a contraction mapping on any bounded set of $(\Cint\cap \Sigma,d)$,
and in particular on the closure (for $\hil$) of the image of $g_s$.
Since, by Proposition~\ref{sp-prop-birkhoff}, $(\Cint\cap \Sigma,d)$ 
is a complete metric space, it follows that $g_s$ admits a 
unique fixed point $x_s\in\Cint\cap \Sigma$.
This implies that $x_s\in\Cint$, $q(x_s)=1$ and 
\begin{equation}\label{sp-supeigenth3}
 h_s(x_s)=h(x_s)+sq(h(x_s))u=\lambda_s x_s\end{equation}
for some $\lambda_s>0$.

For all $s>0$, $h_s:C\to C,\; x\mapsto h(x)+s q(h(x)) u$ 
($q$ and $h_s$ are extended by $0$ at $0$) is continuous, 
order-preserving and homogeneous. Hence,
by~\eqref{sp-supeigenth3} and~\eqref{sp-lem-add51}, 
we get that 
$\cw[C]{h_s}\leq \lambda_s\leq \eigenvalspr[C]{h_s}
\leq \cspr[C]{h_s}\leq \cw[C]{h_s}$, hence 
 $\cw[C]{h_s}= \lambda_s$.
Since $h\leq h_s\leq h_t$ for all $0<s\leq t$, 
Lemma~\ref{sp-lem-add4} implies that 
$\cw[C]{h}\leq \lambda_s\leq \lambda_t$.
Denote $y_k=x_{1/k}$ and $\mu_k=\lambda_{1/k}$.
The sequence $(\mu_k)_{k\geq 1}$ is nonincreasing, thus
it converges towards some real $\mu\geq \cw[C]{h}$.
Since $q(y_k)=1$ for all $k$, the sequence $y_k$ is bounded in 
$(X,\norm{\cdot})$. Then, $\mu y_k-h(y_k)= (\mu-\mu_k) y_k +\frac 1 k q(h(y_k))u$ tends
to $0$ when $k$ goes to $\infty$.

Suppose first that $\mes[C]{h}< \cw[C]{h}$.
Then, 
\[\mes[C]{\frac{1}{\mu} h}<
\cw[C]{h}/\mu\leq 1
\enspace,
\]
and by Lemma~\ref{sp-lem-ksetcontract}, 
$\mu \id -h$ has Property~\PF. 
Hence, $y_k$ has a convergent subsequence,
and since $h$ and $q$ are continuous and $C$ is closed, the limit $y$ of 
this subsequence satisfies $y\in C$, $q(y)=1$ and $h(y)=\mu y$.
Hence, $\mu\leq \eigenvalspr[C]{h}\leq \cspr[C]{h}=\bonsall[C]{h}\leq
\cw[C]{h}\leq \mu$, which implies that
$\bonsall[C]{h}=\cw[C]{h}$ and also proves that
$\eigenvalspr[C]{h}=\cspr[C]{h}$ and that there exists 
$y\in C\setminus\{0\}$ such that $h(y)= \cspr[C]{h} y$.

Suppose now that $h$ is uniformly continuous on bounded sets
of $C$, and assume by contradiction that $\bonsall[C]{h}<\cw[C]{h}$.
Then, $\bonsall[C]{\frac{1}{\mu} h}<\cw[C]{h}/\mu\leq 1$
and by Proposition~\ref{sp-prop1}, $\mu \id -h$ has Property~\PF,
and $0$ is the unique fixed point of $\frac{1}{\mu} h$.
Using Property~\PF, we conclude as above that there exists
an element $y\in C$ such that $q(y)=1$ and $h(y)=\mu y$.
which contradicts the fact that
$0$ is the unique fixed point of $\frac{1}{\mu} h$.
This shows that $\bonsall[C]{h}=\cw[C]{h}$.

Suppose in addition that $\essp[C]{h}<\cw[C]{h}$.  
Then,
$\essp[C]{\frac{1}{\mu} h}< \cw[C]{h}/\mu\leq 1$ and,
by Proposition~\ref{sp-prop1}, $\mu \id -h$ has Property~\PF.
As above, we conclude that there is a vector $y\in C$ such that $q(y)=1$  and $h(y)=\mu y$. 
Therefore, $\mu\leq \eigenvalspr[C]{h}\leq \cspr[C]{h}=\bonsall[C]{h}\leq
\cw[C]{h}\leq \mu$, showing that $h(y)=\cspr[C]{h}y$.

\end{proof}
\begin{corollary}\label{sp-cor-supeigenth}
Let $C$ be a normal cone in a Banach space $(X,\norm{\cdot})$, with
nonempty interior.
Let $h:C\to C$ be a continuous, homogeneous, and order-preserving
map. If, for some $m\geq 1$, we have
\begin{align}
\mes[C]{h^m}< \cw[C]{h^m}
\end{align}
then, 
\begin{align}
\eigenvalspr[C]{h^m}=(\cspr[C]{h})^m=(\bonsall[C]{h})^m= 
\cw[C]{h^m}
\end{align}
and there exists an element $x_m\in C\setminus\{0\}$
such that 
\begin{align}
h^m(x_m)= (\cspr[C]{h})^m x_m\enspace .
\label{sp-supeigenth1new}
\end{align}
\end{corollary}
\begin{proof}
Theorem~\ref{sp-supeigenth0} shows that 
$\eigenvalspr[C]{h^m}=\cspr[C]{h^m}=\bonsall[C]{h^m}= 
\cw[C]{h^m}$, and that there exists an element
$x_m\in C\setminus\{0\}$ satisfying $h^m(x_m)= \cspr[C]{h^m} x_m$.
A straightforward argument in~\cite{Nuss-Mallet} proves
that $\bonsall[C]{h^m}=(\bonsall[C]{h})^m$. Since, by Theorem~\ref{sp-prop-nus1},
$\bonsall[C]{h}=\cspr[C]{h}$, the corollary
is proved.
\end{proof}

\begin{remark}
The Collatz-Wielandt theorem of~\cite{nussbaum86} was recently generalized in~\cite{GV10}
in a different way, to the case of order preserving and positively homogeneous
self-maps of
the interior of a finite dimensional cone with a family of geodesics
in Thompson's metric satisfying Busemann's nonpositive curvature condition.
Some of the conditions of~\cite{GV10} (finite dimension, nonpositive curvature)
are considerably more demanding than the ones of Theorem~\ref{sp-supeigenth0}, 
however, the Collatz-Wielandt type result of~\cite{GV10} remains valid
even if the map cannot be extended continuously to the boundary
of the cone, a property which we require here. 
\end{remark}
\section{Spectral radius of a supremum or of an infimum of functions}
\label{sp-spec-inf}
We now apply the previous characterizations of the spectral radius to show
that the spectral radius acts as a ``morphism'' with respect to
suprema and infima of families of positively homogeneous order preserving
maps satisfying certain selection properties. This is motivated
in particular by zero-sum games, in which Shapley operators
are typically given as infima or suprema of more elementary operators,
see~\cite{mariannemtns,marianne102013} for some applications of the present 
results.

Let $C$ be a proper cone in a Banach space $(X,\norm{\cdot})$, and 
let $(f_a)_{a\in \A}$ be a family of maps $X\to X$.
We say that  $(f_a)_{a\in \A}$ admits an \NEW{upper selection}
if  for all $x\in X$, there exists $a_x\in \A$ such that
$f_{a_x}(x)\geq f_b(x) $ for all $b \in \A$.
We denote by $\sup_{a \in \A} f_a$ the map $f: X\to X$
which associate to $x\in X$ the element $f_{a_x}(x)\in X$,
and call $f$ the \NEW{supremum} of the family $(f_a)_{a\in \A}$.
We define symmetrically the notion of  \NEW{lower selection} of
$(f_a)_{a\in \A}$ and the associated \NEW{infimum} which we denote by
$\inf_{a \in \A} f_a$.

\begin{proposition}\label{sp-lem-add3}
Let $C$ be a proper cone in a Banach space $(X,\norm{\cdot})$, 
and let $(f_a)_{a\in \A}$ be a family of maps $X\to X$ that are
homogeneous.
Assume that  $(f_a)_{a\in \A}$ admits a upper selection
and denote by $f$ its supremum.
Then, $f$ is homogeneous and we have:
\begin{align}
\eigenvalspr[C]{f}&\leq \sup_{a\in \A} 
\eigenvalspr[C]{f_a}\enspace .\label{sp-lem-add32}
\end{align}
Assume further that the maps $f_a$ are 
continuous and order-preserving with respect to $C$, that 
$f$ is continuous and that $C$ is a normal cone. Then,
$f$ is order-preserving, and we have:
\begin{align}
\sup_{a\in \A} \cspr[C]{f_a}\leq \cspr[C]{f}
\leq \max(\mes[C]{f}, \sup_{a\in \A}  \cspr[C]{f_a})\enspace .
\label{sp-lem-add31}
\end{align}
In particular, when $\mes[C]{f}< \cspr[C]{f}$, we have
\begin{align}\label{sp-lem-add33}\eigenvalspr[C]{f}= 
\sup_{a\in \A} \eigenvalspr[C]{f_a}
=\sup_{a\in \A} \cspr[C]{f_a}=\cspr[C]{f}
\enspace ,
\end{align} 
and the suprema are attained in~\eqref{sp-lem-add33}.
\end{proposition}
\begin{proof}
For all $\lambda> 0$ and $x\in X$, we get 
$f_a(\lambda x)=\lambda f_a(x)$ for all $a\in\A$, hence
$f(\lambda x)=f_{a_{\lambda x}}(\lambda x)=
\lambda f_{a_{\lambda x}}(x)\leq \lambda f(x)
=\lambda f_{a_x}(x)=f_{a_x}(\lambda x)\leq f(\lambda x)$,
and $f(\lambda x)=\lambda f(x)$. This shows that
$f$ is homogeneous.
Let now $x\in C$ and 
$\lambda\geq 0$ be such that $f(x)=\lambda x$.
Since $(f_a)_{a\in\A}$ admits a upper
selection, there exists $a_x\in\A$ such that 
$f_{a_x}(x)=f(x)=\lambda x$. It follows that
$\lambda\leq \eigenvalspr[C]{f_{a_x}}
\leq \sup_{a\in\A} \eigenvalspr[C]{f_{a}}$.
Since this holds for all $\lambda$ such that $f(x)=\lambda x$ for
some $x\in C$, we deduce~\eqref{sp-lem-add32}.

Now assume that the maps $f_a$ are 
continuous and order-preserving with respect to $C$, that 
$f$ is continuous and that $C$ is a normal cone. 
Let $x,y\in X$ be such that $x\leq y$. We get 
$f_a(x)\leq f_a(y)$ for all $a\in\A$, hence
$f(x)=f_{a_x}(x)\leq f_{a_x}(y)\leq f(y)$, which shows that
$f$ is order-preserving.
Using Lemma~\ref{sp-lem-add2}, we obtain that
$\cspr[C]{f}\geq \cspr[C]{f_a}$ for all $a\in \A$,
which shows the first inequality in~\eqref{sp-lem-add31}.
Combining the previous inequalities with
the fact that $\eigenvalspr[C]{f_a}\leq \cspr[C]{f_a}$ for
all $a\in\A$, we get:
\begin{align*} \eigenvalspr[C]{f}\leq 
\sup_{a\in \A} \eigenvalspr[C]{f_a}
\leq \sup_{a\in \A} \cspr[C]{f_a}\leq\cspr[C]{f} \enspace .
\end{align*}
Then, the equalities in~\eqref{sp-lem-add33} follow 
from Theorem~\ref{sp-th-nus2}.
Moreover, by~\eqref{sp-th-nus21},
there exists 
$x\in C\setminus\{0\}$ such that $f(x)=\eigenvalspr[C]{f} x$.
Then, $f_{a_x}(x)=\eigenvalspr[C]{f} x$, hence
$\eigenvalspr[C]{f} \leq \eigenvalspr[C]{f_{a_x}}$ and
since $\eigenvalspr[C]{f_{a_x}}\leq \cspr[C]{f_{a_x}}
\leq \cspr[C]{f}=\eigenvalspr[C]{f}$, we get that the suprema are
attained in~\eqref{sp-lem-add33}.
Finally, the second inequality in~\eqref{sp-lem-add31}
follows from~\eqref{sp-lem-add33}.
\end{proof}

\begin{proposition}\label{sp-lem-inf}
Let $C$ be a proper cone with non-empty interior in a Banach space $(X,\norm{\cdot})$, 
and let $(f_a)_{a\in \A}$ be a family of maps $X\to X$ that are
homogeneous.
Assume that  $(f_a)_{a\in \A}$ admits a lower selection
and denote by $f$ its infimum.
Then, $f$ is homogeneous and we have:
\begin{align}
\cw[C]{f}&= \inf_{a\in \A} 
\cw[C]{f_a}\enspace .\label{sp-lem-inf1}
\end{align}
Assume in addition that the maps $f_a$ are 
continuous and order-preserving with respect to $C$, that 
$f$ is continuous and that $C$ is a normal cone. Then,
$f$ is order-preserving, and we have:
\begin{align}
\cspr[C]{f}\leq \inf_{a\in \A} \cspr[C]{f_a}
\enspace .\label{sp-lem-inf2}
\end{align}
When, in addition, either $\mes[C]{f}<\cw[C]{f}$
or $f$ is uniformly continuous on bounded sets of $C$,
the following equalities hold:
\begin{align}
\cspr[C]{f}=\inf_{a\in \A} \cspr[C]{f_a}=
\inf_{a\in \A} \cw[C]{f_a} = \cw[C]{f}
\enspace .
\label{sp-lem-inf3}
\end{align}
\end{proposition}
\begin{proof}
The proof of the properties that $f$ is homogeneous of degree 1, 
or order-preserving are identical to that for Lemma~\ref{sp-lem-add3}.
Also, the proof of~\eqref{sp-lem-inf2} is identical to that of
the first inequality in~\eqref{sp-lem-add31}.
{From} Lemma~\ref{sp-lem-add4}, we get that 
$\cw[C]{f}\leq \inf_{a\in \A} \cw[C]{f_a}$.
Let $\lambda>0$ and $x\in \Cint$ be such that $f(x)\leq \lambda x$.
By asumption, there exists $a_x\in\A$ such that
$f_{a_x} (x)=f(x)\leq \lambda x$,  hence $\lambda\geq 
\cw[C]{f_{a_x}}\geq \inf_{a\in \A} \cw[C]{f_a}$.
since this holds for all $\lambda>0$ and $x\in \Cint$ 
such that $f(x)\leq \lambda x$, we deduce that 
$\cw[C]{f}\geq \inf_{a\in \A} \cw[C]{f_a}$,
which shows~\eqref{sp-lem-inf1}.
Finally, the equalities in~\eqref{sp-lem-inf3} follow
from~\eqref{sp-lem-inf1},~\eqref{sp-lem-inf2},
Lemma~\ref{sp-lem-add5} and Theorem~\ref{sp-supeigenth0}.
\end{proof}

\nocite{goebel,sadovski}
\newcommand{\etalchar}[1]{$^{#1}$}
\def\cprime{$'$} \def\cprime{$'$}

\end{document}